\documentclass[11pt]{article}
\usepackage{amssymb}
\usepackage{mathrsfs}
\usepackage{latexsym,amsmath,amssymb,amsfonts,epsfig,graphicx,cite,psfrag}
\usepackage{eepic,color,colordvi,amscd}
\usepackage{ebezier}
\usepackage{verbatim}
\usepackage{subfigure}
\usepackage{enumitem}
\usepackage{algorithm}
\usepackage{algorithmic}
\usepackage{amsthm}
\usepackage{comment}
\usepackage{color}
\usepackage{longtable}
\theoremstyle{plain}
\newtheorem{theo}{Theorem}[section]

\newtheorem{lem}[theo]{Lemma}

\newtheorem{prop}[theo]{Proposition}

\newtheorem{exmp}[theo]{Example}

\theoremstyle{definition}
\newtheorem{defi}[theo]{Definition}

\theoremstyle{remark}
\newtheorem{rem}[theo]{Remark}
\definecolor{blue}{rgb}{0,0,0.9}
\definecolor{red}{rgb}{0.9,0,0}
\definecolor{green}{rgb}{0,0.9,0}

\def\<{\left\langle}
\def\>{\right\rangle}

\def\M{\mathcal{M}}
\def\T{\mathcal{T}}
\def\D{{\rm D}}
\def\DD{{\rm Diag}}
\def\A{\mathcal{A}}
\def\B{\mathcal{B}}
\def\K{\mathcal{K}}
\def\L{\mathcal{L}}
\def\F{\mathcal{F}}
\def\R{\mathbb{R}}

\def\U{\tilde{U}}
\def\Re{\mathcal{R}}

\def\g{{\rm grad}}

\def\N{\mathcal{N}}

\def\S{\mathbb{S}}
\def\di{{\rm dim}}
\def\rr{{\rm rank}}
\def\dist{{\rm dist}}
\def\dd{{\rm diag}}
\def\P{{\rm Proj}}
\def\O{\mathcal{O}}
\def\OB{{\rm OB}}
\def\St{{\rm St}}
\def\({\left(}
\def\){\right)}
\def\sgn{\texttt{sgn}}
\def\inprod#1#2{\big\langle #1,\,#2\big\rangle}
\def\norm#1{\|#1\|}
\def\Ron{\texttt{Round}}
\let\svthefootnote\thefootnote
\newcommand\blankfootnote[1]{%
	\let\thefootnote\relax\footnotetext{#1}%
	\let\thefootnote\svthefootnote%
}

\def\barR{\bar{R}}
\def\gR{g_{\mbox{\tiny $R$}}}

\begin{document}
	\title{Solving graph equipartition SDPs on an algebraic variety}
	\author{Tianyun Tang
\thanks{Department of Mathematics, National
         University of Singapore, Singapore
         119076 ({\tt ttang@u.nus.edu}).
         }, \quad 
	 Kim-Chuan Toh\thanks{Department of Mathematics, and Institute of 
Operations Research and Analytics, National
         University of Singapore, 
       Singapore
         119076 ({\tt mattohkc@nus.edu.sg}).  The research of this author is supported
by the Ministry of Education, Singapore, under its Academic Research Fund Tier 3 grant call (MOE-2019-T3-1-010).}
	 }

	\date{\today}
	\maketitle

\begin{abstract}
Semidefinite programs are generally challenging to solve due to their high dimensionality. Burer and Monteiro developed a non-convex approach to solve linear SDP problems by applying its low rank property. Their approach is fast because they used factorization to reduce the problem size. In this paper, we focus on solving the SDP relaxation of a graph equipartition problem, which involves an additional semidefinite upper bound constraint over the traditional linear SDP. 
By applying the factorization approach, we get a non-convex problem with an additional non-smooth spectral inequality constraint. We discuss when the non-convex problem is equivalent to the original SDP, and when a second order stationary point of the non-convex problem is also a global minimum. Our results generalize previous works on smooth non-convex factorization approaches for linear SDP to the non-smooth case. Moreover, the constraints of the non-convex problem involve an algebraic variety with some conducive properties that allow us to use Riemannian optimization techniques and non-convex augmented Lagrangian method to solve the SDP problem very efficiently
with certified global optimality.
\end{abstract}

\bigskip
\noindent{\bf keywords:} {Graph equipartition, Burer and Monteiro method, low rank SDP, algebraic variety, Riemannian optimization, augmented Lagrangian method} 
\\[5pt]

\section{Introduction}
\subsection{Literature Review}
Many combinatorial optimization problems such as max-cut and maximum stable set problems are NP-hard, which means that it is almost impossible to solve them exactly in polynomial time. 
These problems are often solved approximately via their semidefinite programming (SDP) relaxations 
which typically result in linear SDP problems of the following form:
\begin{equation}\label{SDP}
{\rm SDP:\ } \min_{X\in \S^{n}} \left\{ \<C,X\>:\ \A \(X\)=b,\ X\succeq 0\right\},
\end{equation}
where $b\in \R^m$ and $C\in \S^n$ are given data, and $\A :\S^n\rightarrow \R^m$ is a given linear mapping (see \cite{combsdp, colorsdp, SDPrelax} for examples of SDP relaxation).

The above linear SDP problem (\ref{SDP}) is a convex programming problem
which can be solved by various well developed solvers such as SDPT3 \cite{SDPT31,SDPT32},
MOSEK \cite{mosek}, and SDPNAL \cite{SDPNALp,SDPNAL}. Those solvers are efficient and accurate if $n$ is moderate (say less than $2000$) and
in the case of SDPT3 and MOSEK, $m$ is also not too large (say less than $20000$). However, if $n$ is large, the SDP will be challenging to solve by the aforementioned solvers because they will incur excessive computing cost and memory usage due to the high dimensionality in $n$. To overcome this difficulty, Burer and Monteiro \cite{BM1,BM2}
proposed to solve the linear SDP \eqref{SDP} by applying the factorization $X=RR^\top$
to reformulate it as the following non-convex problem:
\begin{equation}\label{SLR}
{\rm SDPLR:\ } \min\left\{ \inprod{C}{RR^\top} :\ \A \big(RR^\top\big)=b,\ R\in \R^{n\times r}\right\}.
\end{equation}
Due to the low rank property (see, e.g., \cite{inirank,FL1976,B1995}) of the SDP \eqref{SDP},  which states  that if the feasible region of \eqref{SDP} is compact, then it will have an optimal solution of rank $\leq \sqrt{2m}.$ So \eqref{SDP} is equivalent to \eqref{SLR} when $r\geq \sqrt{2m}$. Burer and Monteiro used an augmented Lagrangian method to solve \eqref{SLR} and its high efficiency was verified in numerical comparison to other algorithms such as the spectral bundle \cite{SB} and interior point methods \cite{interior}.

Since (\ref{SLR}) is a non-convex problem, it may have spurious local minima. As far as we know, the first foundational and thorough analysis about the global optimality of (\ref{SLR}) is done in a series of works \cite{Boumal1,Boumal2,Boumal3,Boumal4}, 
by Boumal et al. %They analysed the global optimality of \eqref{SLR}. 
They proved that under LICQ-like regularity assumptions on the constraints of \eqref{SLR}, if $r\geq \sqrt{2m}$, then for almost all  $C\in \S^n$ except for a set of measure zero, every second order stationary point of \eqref{SLR} is a global optimal solution. This result explains why the BM approach can almost always find the global minimum in practice. They also did a smoothed analysis for problem (\ref{SLR}) and its penalty version, which was useful in choosing the stopping criterion for the augmented Lagrangian subproblems.  Moreover, with their regularity assumption, the set $\left\{ R\in \R^{n\times r}:\ \A \(RR^\top\)=b\right\}$ is a Riemannian manifold embedded in $\R^{n\times r},$ so they can use an optimization algorithm on manifold to solve \eqref{SLR}. They also developed a widely used toolbox called manopt \cite{manopt} to handle
optimization problems on a variety of manifolds. 

One specific example where Riemannian algorithms are expected to be
efficient for solving
the problem (\ref{SLR}) is the max-cut problem for which $\A (X) = \dd(X)$, and the underlying manifold is the oblique manifold:
$$\OB_{n,r}:=\left\{ R\in \R^{n\times r}:\ \dd\big(RR^\top\big)=e\right\},$$  whose tangent space and retraction mapping have a simple formula. In this case, Boumal et al. \cite{Boumal2} used manopt to solve (\ref{SLR})  and verified its high efficiency as compared to Burer and Monteiro's augmented Lagrangian method and other convex algorithms. 

The research by Boumal et al. in \cite{Boumal1,Boumal2,Boumal3,Boumal4} 
was a milestone in 
studying the 
global optimality of the non-convex factorized formulation of a linear SDP problem
by combining the geometric and low-rank property of the SDP problem.
Their results have since been generalized by Cifuentes in \cite{Cifu} to linear SDP problems with 
multiple positive semidefinite block variables and inequality constraints.
One of our goals in this paper is to generalize the theory of Boumal et al. to a class of SDP problems
with an additional semidefinite upper bound constraint, where the feasible sets
of the non-convex factorized
models are no longer smooth manifolds.

\subsection{Our Contribution}
In this paper, we consider the graph equipartition problem, that is, given a graph $G$ of size $n$, we want to partition it into $k$ parts of equal size $q:=n/k$, such that the total number of edges between different parts is minimized.
Its SDP relaxation is given as follows (see \cite{SDPrelax,SDPgep}):
\begin{equation}\label{GEP}
{\rm GEP:}\min\left\{ \<L,Y\>:\ \dd(Y)=e,\ Ye = qe,\ qI\succeq Y\succeq 0\right\},
\end{equation}
where $L$ is the Laplacian of the graph and $e\in \R^n$ is the vector of all ones. By changing the variable $X= \left(Y-\frac{1}{k}ee^\top\right)\frac{k}{k-1}$, problem (\ref{GEP}) is equivalent to the following SDP problem,
\begin{equation}\label{GEP1}
{\rm GEP1:}\min\left\{ \<L,X\>:\ \dd(X)=e,\ Xe=0,\ X\succeq 0,\ \frac{n}{k-1}I\succeq X\right\}.
\end{equation}
Note that for $k=2$, we can remove the redundant constraint $\frac{n}{k-1}I\succeq X,$ and it becomes a standard linear SDP\footnote{Note that one may replace $Xe=0$	by $\<X,ee^\top\>=0$ since $X\succeq 0$.}. To solve (\ref{GEP1}), we consider the following more general problem:
\begin{equation}\label{SDP1}
{\rm SDP1:}\ \min\left\{ \<C,X\>:\ \A\(X\)=b,\ Xe=0,\ \inprod{X}{ee^\top} =0,\ \alpha I\succeq X\succeq 0\right\},
\end{equation}
where $\A:\ \S^n\rightarrow \R^{m-1}$ is a general linear mapping and $\alpha>0.$ Note that we add the redundant constraint $Xe=0$ to facilitate our theoretical analysis later. For later usage, we 
define 
\begin{equation}
\Omega_\alpha^1:=\left\{X\in \S^n:\ \alpha I\succeq X\succeq 0\right\}. \tag{5a}
\end{equation}

The low-rank factorization model of (\ref{SDP1}) is as follows:
\begin{equation}\label{SDPLR}
{\rm SDPLR1:\ } \min\left\{ \frac{1}{2}\inprod{C}{RR^\top}
:R\in \R^{n\times r},\A\big(RR^\top\big)=b,\ R^\top e=0\ ,\|R\|_2\leq \sqrt{\alpha}\right\},
\end{equation}
where $\norm{R}_2$ is the spectral norm of $R$.
We add the factor $\frac{1}{2}$ in the objective function for convenience. Note that the main difference between (\ref{SDP}),(\ref{SLR}) and (\ref{SDP1}),(\ref{SDPLR}) is that we have an extra inequality constraint $\|R\|_2\leq \sqrt{\alpha}$. We will prove that (\ref{SDPLR}) is equivalent to (\ref{SDP1}) if $r$ is larger than some given bound, which is a generalization of Burer and Monteiro's result. We will also prove that under the 
constraint nondegeneracy condition (see \cite{Sensitivity}),
%%which is related to Robinson's constraints qualification \cite{Robin}, 
for almost all $C\in \S^n$ except for a set of measure zero, any second order stationary point of (\ref{SDPLR}) is also a global optimal solution. This is a generalization of the results in \cite{Boumal2,Boumal3,Boumal4}. 

If we choose $\A(\cdot) =\dd(\cdot),$ $b=e$ and $C=L$, then (\ref{SDPLR}) is the low-rank factorization of (\ref{GEP1}) for $\alpha={\frac{n}{k-1}}$, and we get
\begin{equation}\label{GEPLR}
 \min\left\{ \frac{1}{2}\inprod{C}{RR^\top}:
 R\in \R^{n\times r},\dd\big(RR^\top\big)=e,\ R^\top e=0\ ,\|R\|_2\leq \sqrt{\alpha}\right\},
\end{equation}
whose feasible region is the intersection of the following two sets:
\begin{equation}\label{BCmani}
\B_{n,r}:=\left\{ R\in \R^{n\times r}:\ \dd\big(RR^\top\big)=e,\ R^\top e=0\right\},
\end{equation}
\begin{equation}
\Omega_\alpha^2:=\left\{ R\in \R^{n\times r}:\ \|R\|_2\leq \sqrt{\alpha}\right\}.
\end{equation}
Note that if $X=RR^\top$, then $X\in \Omega_\alpha^1$ if and only if $R\in \Omega_\alpha^2.$ The set $\B_{n,r}$ is obtained by adding the constraint $R^\top e=0$ to the oblique manifold $\OB_{n,r}$. Just as each element of $\OB_{n,r}$ can be viewed as the Cartesian product of $n$ unit vectors, which are freely chosen, each element of $\B_{n,r}$ can be considered as the Cartesian product of $n$ unit vectors, but they are constrained to have their sum equal to zero. Therefore, those vectors are no longer independent from each other and $\B_{n,r}$ is not as simple as $\OB_{n,r}.$ In fact, later we will show that $\B_{n,r}$ is not even a manifold, but an algebraic variety defined as the set of common zeros of a system of polynomials (see \cite{AlgV} chapter 3 definition 1). However, even though it is not a manifold, we can still derive an explicit formula to describe the tangent cones and second order tangent sets of its singular points.
 More importantly, we find that one possible retraction mapping of $\B_{n,r}$ is related to the well known geometric median problem, which means that we can compute the retraction efficiently. These findings are crucial for our algorithmic design. We will design an algorithm that equips the underlying Riemannian optimization method with a strategy to check the global optimality of a singular point and escape any non-optimal singular point efficiently. To solve (\ref{GEPLR}) with the spectral upper bound, we use an augmented Lagrangian method on the algebraic variety $\B_{n,r}$, that is, we keep the primal iterations on $\B_{n,r}$ while penalizing $\|R\|_2\leq \sqrt{\alpha}.$ Now, we summarize our contributions as follows:
\begin{itemize}
\item We study low rank factorization of graph equipartition SDP and prove the equivalence between the factorized problem and the original convex problem in terms of global optimality. We also prove that under the constraint nondegeneracy condition \cite{Sensitivity}, any 
second order stationary point of the factorized problem is a global optimal solution with probability 1. These results generalize the results of Boumal et al. from \cite{Boumal1,Boumal2,Boumal3,Boumal4} since we have an extra non-smooth inequality constraint and our feasible region is not necessarily a smooth manifold. Our study here may serve as a prototype for extending the 
non-convex factorization approach to other types of SDP problems whose
resulting feasible regions may not be smooth manifolds.

\item We study a special algebraic variety $\B_{n,r}$, which is the set of zeros of the equality constraints of the factorized problem. We study the local geometric properties of $\B_{n,r}$ on both smooth points and singular points. We also show that retraction is equivalent to the geometric median problem under certain conditions. These properties allow us to conduct manifold optimization methods on the smooth part of $\B_{n,r}$ while escaping from non-optimal singular points of $\B_{n,r}$.

\item We develop a gradient descent method and an augmented Lagrangian method on an algebraic variety to solve graph equipartition SDP problems and conduct numerical experiments to demonstrate their high efficiency. Although manifold structures are increasingly used in algorithms
for solving SDP problems (see \cite{Boumal2,maniALM,WZW}), our novelty here is in developing algorithms to solve the SDP problems 
on an algebraic variety that is not necessarily a smooth manifold.
\end{itemize}

\subsection{Organization of this Paper}
The paper is organised as follows, in section 2, we prove the equivalence of the upper bounded linear SDP problem (\ref{SDP1}) and the low-rank factorized model (\ref{SDPLR}) under a certain regularity condition. We also establish a sufficient condition for a
second order stationary point of (\ref{SDPLR}) to be a global optimal solution. In section 3, we study the analyse the local geometric properties of $\B_{n,r}$. We show that its projection mapping and retraction mapping around smooth points can be computed efficiently. In section 4, we will design algorithms to solve graph equipartition SDP problems. In section 5, we conduct
 numerical experiments to evaluate the performance of our proposed approach and algorithm. In section 6, we give a brief conclusion. We show how to deal with the singular points in $\B_{n,r}$ in the appendix. Some auxiliary results and proofs are also put in the appendix.

\subsection{Notation}
Note that we often omit stating the dimension of a vector or matrix if it is already clear from the context. We use the following notation.
\begin{itemize}
\item[(1)] $e\in \R^n$ is the all ones vector and $J:=I-ee^\top/n.$ 
\item[(2)] $e_i$ is a column vector such that its $i$th entry is 1 and all other entries are zero. We will omit the dimension of $e_i$ if it is already clear.
\item[(3)] ${\rm S}^{n-1}:=\left\{x\in \R^n:\ \|x\|_2=1\right\}$ is the unit sphere, 
and $\St(n,p):=\left\{ X\in \R^{n\times p}:\ X^\top X=I\right\}$ is the Stiefel manifold.
\item[(4)] $\S^n$ and $\S^n_+$ denote the space of $n\times n$ symmetric matrices
and its subset of positive semidefinite matrices, respectively. 
 For notational simplicity, we use $X\succeq 0$ to denote $X\in \S_+^n$.
 We denote  the trace of a square matrix $X$ as ${\rm Tr}(X).$ 
We use $\|X\|$ or $\|X\|_F$ to denote the Frobenius norm of a given matrix $X.$
\item [(5)] For $x\in \R^n$, $\dd(x)\in \R^{n\times n}$ is the diagonal matrix with diagonal entries $x$. For $X\in \R^{n\times n}$, $\dd(X)\in \R^n$ is the diagonal vector of $X$, and $\DD(X)=\dd(\dd(X)).$ 
\item [(6)] $\P_F(X)$ is the orthogonal projection of $X$ onto $F,$ which could be a linear space or a compact set. $\P_{F^\perp}(X)=X-\P_{F}(X).$ 
\item [(7)] For $A,B\in \R^{m\times n}$, $A\circ B\in \R^{m\times n}$ is defined by $(A\circ B)_{ij}=a_{ij}b_{ij}.$
\end{itemize}

\subsection{Preliminaries.}

In this paper, we will frequently use the definition of tangent cone (see \cite{Perturbation} section 2). For any closed set $S\subset \R^n$ and $x\in S,$ the inner tangent cone $\T^i_S(x)$ is defined as follows:
\begin{equation}\label{tancone}
\T^i_S(x):=\left\{ h\in \R^n:\ \dist\(x+th,S\)=o(t),\ t\geq 0\right\}.
\end{equation}
The contingent tangent cone is defined as follows:
\begin{equation}\label{ctancone}
\T_S(x):=\left\{ h\in \R^n:\ \exists t_k \downarrow 0,\ \dist\(x+t_kh,S\)=o(t_k)\right\}.
\end{equation}
It is easy to see that $\T^i_S(x)\subset \T_S(x).$ Note that there is also another tangent cone called Clarke tangent cone (see \cite{Perturbation} section 2), but for a closed convex set $S$ that is non-singleton, these three cones are equal to each other (see \cite{Perturbation} proposition 2.55-2.57). Therefore, when we deal with a closed convex set $S$, we simply use $\T_S(x)$ to denote the tangent cone.

In order to understand more about the local geometric property of a given closed set $S\subset \R^n$, we need second order tangent sets which are defined as follows (see \cite{Perturbation} definition 3.28):
\begin{equation}\label{istan}
\T^{i,2}_S(x,h):=\left\{ w\in \R^n:\ \dd\Big(x+th+\frac{1}{2}t^2w,S\Big)=o(t^2),t\geq 0\right\},\ {\rm for}\ h\in \T^i_S(x),
\end{equation}
\begin{equation}\label{ostan}
\T^{2}_S(x,h):=\left\{ w\in \R^n:\ \exists t_k\downarrow 0,\ \dd\Big(x+t_kh+\frac{1}{2}t_k^2w,S\Big)=o(t_k^2)\right\},\ {\rm for}\ h\in \T_S(x),
\end{equation}
where the first and second sets are called inner second order tangent set and outer second order tangent set, respectively.

%%%%%%%%%%%%%%%%%%%%%%%%%%%%%%%%
\section{Relation between SDP1 and SDPLR1}
\subsection{Low-rank property of SDP1}
The following theorem states the low-rank property of a linear SDP with an additional
semidefinite upper bound constraint of the form $\alpha I \succeq X$. This result here generalizes the results in \cite{inirank,  BM1,BM2}.

\begin{theo}\label{LR}
Consider the following SDP\footnote{Note that we use a general affine constraint $\B \(X\)=b$ to replace the equality constraints in (\ref{SDP1}), excluding the redundant constraint $Xe=0$.}:
\begin{equation}\label{ISDP}
\min\left\{ \<C,X\>:\ \B(X)=b,\ X\succeq 0,\ \alpha I\succeq X\right\},
\end{equation}
where $\alpha>0,$ $b\in \R^m$ has dimension $m\geq 1$, 
and $\B:\S^n \rightarrow \R^m$ is a linear mapping such that its adjoint $\B^*$ is injective\footnote{One can remove this condition and replace the rank bound $\lceil\frac{\beta}{\alpha}\rceil+\sqrt{2m}-1$ with $\lceil\frac{\beta}{\alpha}\rceil+\sqrt{2\cdot \rr\(\B\)}-1,$ where $\rr\(\B\)$ is the dimension of the range space of the linear mapping $\B.$ The proof is similar.}.
If (\ref{ISDP}) has a feasible solution and 
$$
\beta:=\max\left\{ {\rm Tr}(X):\ \B\(X\)=b,\ \alpha I \succeq X\succeq 0
\right\},$$
 then (\ref{ISDP}) has a solution $X$ such that $\rr\(X\)\leq \lceil\frac{\beta}{\alpha}\rceil+\sqrt{2m}-1.$
\end{theo}
\begin{proof}
Since the feasible region is non-empty and compact,  an optimal solution exists. Suppose $X$ is an optimal solution such that it has the maximal\footnote{Since the number of eigenvalues equal to $\alpha$ is a bounded integer, we can always choose such an $X$ out of the optimal solution set.} number of eigenvalues equal to $\alpha$, say $\ell$. Let $r=\rr\(X\).$ Consider the eigenvalue decomposition $X=P\Lambda P^\top ,$ with 
$\Lambda = {\rm diag}(\lambda)$ and the eigenvalues $\lambda$ are arranged such that
$\lambda_1=\lambda_2=\ldots=\lambda_\ell=\alpha>\lambda_{\ell+1}\geq\ldots\geq \lambda_r>0=\lambda_{r+1}=\ldots=\lambda_n.$ 

It is easy to see that we have $\ell \alpha\leq {\rm Tr}(X)\leq \beta,$ so $\ell\leq \frac{\beta}{\alpha}.$ If $\ell=\frac{\beta}{\alpha}$, then $\rr\(X\)=\ell=\frac{\beta}{\alpha}\leq \lceil \frac{\beta}{\alpha}\rceil+\sqrt{2m}-1,$ which satisfies the bound. If $\ell<\frac{\beta}{\alpha}$, then $\ell\leq \lceil\frac{\beta}{\alpha}\rceil-1.$ Now, suppose $r>\lceil\frac{\beta}{\alpha}\rceil+\sqrt{2m}-1$, then $s:=r-\ell>\sqrt{2m}$. Let $\Lambda_1=\dd\big(\(\lambda_{\ell+1},\ldots,\lambda_r\)^\top\big),$ and $P_1$ be its corresponding matrix of eigenvectors. Consider the following linear system:
\begin{equation}\label{LS}
\B\big(P_1SP_1^\top\big)=\vec{0},\ S\in \S^s
\end{equation}
Because $\di\(\S^s\)=\frac{s(s+1)}{2}>m$, the linear system has a solution $S\neq {\bf 0}.$ Consider $X_t=X+tP_1SP_1^\top.$ We have $\B\(X_t\)=\B\(X\)=b$ for any $t\in \R.$ For $t\neq 0$ sufficiently small, $X_t$ is still in the feasible region of (\ref{ISDP}). From the optimality of $X$ and varying $t$ over slightly positive and negative values, we get $\<C,P_1SP_1^\top\>=0$.
Thus $\<C,X_t\>=\<C,X\>$ for any $t\in \R$, and $X_t$ is an optimal solution for all $t$  sufficiently small.
Let $\mathcal{T}:=\left\{t\in \R:\ X_t\succeq 0,\ \alpha I\succeq X_t\right\}.$ Since $S\neq 0$, $\mathcal{T}$ is a bounded closed interval $[t_1,t_2]$. We consider $X_{t_1}.$ It has eigenvalues $\lambda_1'=\lambda_2'=\ldots=\lambda_\ell'=\alpha\geq \lambda_{\ell+1}'\geq \ldots\geq \lambda_r'\geq 0=\lambda_{r+1}'=\ldots=\lambda_n'=0.$ Since $t_1$ is on the boundary of $\mathcal{T}$, either $\lambda_{\ell+1}'=\alpha$ or $\lambda_r'=0.$ However, since $X$ has the maximum number of eigenvalues equal to $\alpha$, $\lambda_{\ell+1}'<\alpha$ and we have $\lambda_r'=0.$ Therefore, we have found another optimal solution $X_{t_1}$ whose number of eigenvalues equal to $\alpha$ is still maximal but has rank less than $\rr\(X\).$ We can continue this operation from $X_{t_1}$ until we find an optimal solution with rank $\leq \lceil\frac{\beta}{\alpha}\rceil+\sqrt{2m}-1.$
\end{proof}

If we do not have the constraint $\alpha I\succeq X$ in Theorem \ref{LR} but $\left\{X\in \S^n:\ X\succeq 0,\ \B (X)=b\right\}$ is compact, then we can choose a sufficiently large $\alpha$ such that 
any point $X$ in the compact feasible region satisfies ${\rm Tr}(X) \leq \alpha$ (which also 
implies that $X \preceq \alpha I$).
In this case, the constant $\beta$ in Theorem \ref{LR} is at most $\alpha$ and  
then the rank-bound is $\sqrt{2m},$ which is almost the same as the classical bound of $\frac{1}{2}\(\sqrt{8m+1}-1\).$

Let us now consider (\ref{SDP1}), we can remove $Xe=0$ since it is redundant. Then there are $m$ affine constraints. From Theorem~\ref{LR}, there exists an optimal solution of rank $\leq \lceil\frac{\beta}{\alpha}\rceil+\sqrt{2m}-1.$ This means that (\ref{SDP1}) is equivalent to (\ref{SDPLR}) as long as $r\geq \lceil\frac{\beta}{\alpha}\rceil+\sqrt{2m}-1.$ 
For (\ref{GEP1}), since $\dd(X)=e$, we may choose $\beta=n$ and the rank bound is $k+\sqrt{2(n+1)}-2.$

%%%%%%%%%%%%%%%%%%%%%%%%%%%%%%%%%%%%%%%%%%%%%%%%%%%%%%
%%%%%%%%%%%%%%%%%%%%%%%%%%%%%%%%%%%%%%%%%%%%%%%%%%%%%%
\subsection{A sufficient condition for a {second order stationary point} of SDPLR1 to be globally optimal}

Here we establish a sufficient condition for a second order stationary point of the 
non-convex problem SDPLR1 to be globally
optimal. It turns out that the condition is closely related to the 
constraint nondegeneracy condition of the problem SDPLR1 (\ref{SDPLR}).
Thus in the first part of this subsection, we elaborate on the notation for matrix decomposition and several tangent cones and normal cones. In the second part, we state the regularity condition and its equivalence to the constraint nondegeneracy condition of 
(\ref{SDPLR}). In the third part, we prove that under the regularity condition, a rank-deficient {second order stationary point} of (\ref{SDPLR}) is also a global optimal solution. In the last part, we show that the rank-deficient condition usually holds provided the parameter $r$ in (\ref{SDPLR}) is large enough.

Let $R\in \R^{n\times r}$ be a feasible solution of (\ref{SDPLR}). Then $RR^\top$ is a feasible solution of (\ref{SDP1}). From now on, we suppose $R$  has the following singular value decomposition (SVD)\footnote{These notation for the SVD will be used frequently later.} 
\begin{eqnarray}\label{SVD}
R=\begin{pmatrix} U_1&U_2&U_3\end{pmatrix}\begin{pmatrix}\sqrt{\Lambda_1}& &\\&\sqrt{\Lambda_2}&\\ &&{\bf 0}\end{pmatrix}\ \begin{pmatrix} V_1^\top\\V_2^\top\\V_3^\top\end{pmatrix},
\end{eqnarray}
where $U_1\in \R^{n\times r_1}$, $U_2\in \R^{n\times r_2}$, $U_3\in \R^{n\times (n-r_1-r_2)}$ and $\begin{pmatrix}U_1&U_2&U_3\end{pmatrix}\in \St(n,n)$; $V_1\in \R^{r\times r_1},$ $V_2\in \R^{r\times r_2},$ $V_3\in \R^{r\times (r-r_1-r_2)}$ and $\begin{pmatrix}V_1&V_2&V_3\end{pmatrix}\in \St(r,r)$; $\sqrt{\Lambda_1}$ corresponds to singular values equal to $\sqrt{\alpha},$ $\sqrt{\Lambda_2}$ corresponds to singular values in the interval $(0,\sqrt{\alpha})$; ${\bf 0}$ is the zero matrix in  $\R^{(n-r_1-r_2)\times (r-r_1-r_2)}.$  From $R^\top e=0$ and hence $U_1^\top e=0$, $U_2^\top e=0$, we have $$J\(U_1,U_2,R\)=\(U_1,U_2,R\).$$

Based on the SVD of $R$, we know that $X=RR^\top$ has the following eigenvalue decomposition,
$$ X=\begin{pmatrix} U_1&U_2&U_3\end{pmatrix}\begin{pmatrix}\Lambda_1& &\\&\Lambda_2&\\ &&\hat{\bf 0}\end{pmatrix}\ \begin{pmatrix} U_1^\top\\U_2^\top\\U_3^\top\end{pmatrix},$$ where $\hat{\bf 0}\in \R^{(n-r_1-r_2)\times (n-r_1-r_2)}.$ From these decompositions, we have $\rr(R)=\rr(X)=r_1+r_2.$

We let $\T_{\Omega_\alpha^1}(X)$, $\T_{\Omega_\alpha^2}(R)$, $\N_{\Omega_\alpha^1}(X)$, $\N_{\Omega_\alpha^2}(R)$ be their tangent cones and normal cones\footnote{Note that the standard notation for normal cone is $\N_{\T_{\Omega_\alpha}\(X\)}\(Y\).$ As we always consider $Y=0,$ so we choose a more convenient notation. }. From proposition 2.5 and proposition 2.9 in \cite{Firstcone} (also see \cite{Secondcone}), if $\lambda_{\max}\(X\)=\alpha$ and $\lambda_{\min}\(X\)=0,$ then for any $H\in \S^n$, $Y\in \R^{n\times r},$ 
we have the following formulas for various directional derivatives:
$$\lambda_{\max}'\(X;H\)=\lambda_{\max}\big(U_1^\top HU_1\big),\ \lambda_{\min}'\(X;H\)=\lambda_{\min}\big(U_3^\top HU_3\big)
$$ and 
$$\sigma_{\max}'(R;Y)=
\lambda_{\max}\( \frac{U_1^\top YV_1+V_1^\top Y^\top U_1}{2}\).
$$
Thus, we have\footnote{Those cones are also correct if $\lambda_{\max}\(X\)<\alpha$ or $\lambda_{\min}\(X\)>0.$} 
\begin{eqnarray*}
\T_{\Omega_\alpha^1}\(X\)
&=&\left\{ H\in \S^n:\ U_1^\top H U_1\preceq 0,\ U_3^\top H U_3\succeq 0\right\},
\\[5pt]
\T_{\Omega_\alpha^2}\(R\)
&=&\left\{ Y\in \R^{n\times r}:\ U_1^\top YV_1+V_1^\top Y^\top U_1\preceq 0\right\},
\\[5pt]
\N_{\Omega_\alpha^1}\(X\)&=&\left\{ U_1 S U_1^\top-U_3 W U_3^\top:\ S\in \S^{r_1}_+,\ W\in \S^{n-r_1-r_2}_+\right\},
\\[5pt]
 \N_{\Omega_\alpha^2}\(R\)&=&\left\{ U_1ZV_1^\top:\ Z\in \S^{r_1}_+\right\}.
\end{eqnarray*}

Since (\ref{SDP1}) is a convex problem, the following condition is a sufficient 
optimality condition for SDP1.
\begin{prop}\label{KKTcvx}
If $\hat{X}$ is feasible for (\ref{SDP1}) and there exists $y\in \R^{m-1}$ such that $J\(C-\A^*(y)\)J \in -\N_{\Omega_\alpha^1}(\hat{X}),$ then $\hat{X}$ is an optimal solution of (\ref{SDP1}).
\end{prop}
\begin{proof}
Define $\tilde{\A}: \S^n\rightarrow \R^{m-1}\times \R^n\times \R $ such that $\tilde{\A}\(X\)=\(\A(X),Xe,\<X,ee^\top\>\).$ Then $\tilde{\A}^*:\R^{m-1}\times \R^n\times \R\rightarrow \S^n$ is given by $\tilde{\A}^*(y,z,\beta)=\A^*y
+\frac{1}{2}(ez^\top+z e^\top)+\beta ee^\top.$ Define $\tilde{b}:=\(b,\vec{0},0\).$ Then (\ref{SDP1}) can be simplified as:
$$ 
\min\left\{ \<C,X\>:\ \tilde{\A}(X)=\tilde{b},\ X\in \Omega_\alpha^1\right\}.
$$ 
Note that $J\(C-\A^*(y)\)J \in -\N_{\Omega^1_\alpha}(\hat{X})$ implies that there exists $\tilde{y}\in \R^{m-1}\times \R^n\times \R$ such that $C-\tilde{\A}^*(\tilde{y})\in -\N_{\Omega_\alpha^1}(\hat{X}).$ Now, suppose $\hat{X}$ is not an optimal solution, then there exists $\bar{X}\in \S^n$ such that $\bar{X}\in \Omega_\alpha^1,$ $\tilde{\A}(\bar{X})=\tilde{b}$ and $\<C,\bar{X}\><\inprod{C}{\hat{X}}.$ Then we have $\tilde{\A}(\hat{X}-\bar{X})=0$ which implies that $\inprod{\tilde{\A}^*(\tilde{y})}{\hat{X}-\bar{X}}=0.$ Hence $\inprod{C-\tilde{\A}^*(\tilde{y})}{\hat{X}-\bar{X}}>0.$ 
Since $ C-\tilde{\A}^*(\tilde{y})\in -\N_{\Omega_\alpha^1}(\hat{X})$ and $\bar{X}\in {\Omega_\alpha^1}(\hat{X})$, 
we get by the definition of the normal cone that $ \inprod{ C-\tilde{\A}^*(\tilde{y})}{
\hat{X}-\bar{X}} \leq 0,$ which is a contradiction.
Thus $\hat{X}$ must be an optimal solution of \eqref{SDP1}.
\end{proof}

\medskip
Now, we define the mapping $G:\ \R^{n\times r}\rightarrow \R^{m-1}\times \R^{r}\times \R^{n\times r}\ {\rm such\ that}$
$$
G(R):=\Big(\frac{\A\(RR^\top\)-b}{2},R^\top e,R\Big).
$$ Let $\K:=\{0\}^{m-1}\times\{0\}^r\times \Omega_\alpha^2.$ Then the problem SDPLR1
(\ref{SDPLR}) can be simplified as 
\begin{equation}\label{Oct_25_1}
\min\left\{\frac{1}{2}\inprod{C}{RR^\top}:
\ G(R)\in \K\right\}. {\rm\tag{6\,$^\prime$}}
\end{equation}
Robinson's constraints qualification (CQ) \cite{Robin} is said to hold at $R$ for (\ref{Oct_25_1}) if 
\begin{equation}\label{RobinCQ}
0\in {\rm int}\left\{ G(R)+G'(R)\R^{n\times r}-\K\right\}.
\end{equation}

Note that by Corollary 2.98 in \cite{Perturbation}, the above condition is equivalent to 
\begin{eqnarray}\label{RobinCQ2}
G'(R)\R^{n\times r} - \T_{\K}(G(R)) = \R^{m-1}\times\R^r\times \R^{n\times r}
\end{eqnarray}
where $\T_\K(G(R))$ is the tangent cone to $\K$ at $G(R).$

For any $R\in \R^{n\times r}$, define the linear mapping 
$\gR:\R^{n\times r}\rightarrow \R^{m-1}\times \R^{r}$ by 
$$\gR(H)=\Big(\frac{1}{2}{\A\big( RH^\top+HR^\top\big)},H^\top e\Big).$$
 Then $\gR^*(y,z)=\A^*(y)R+ez^\top.$ Note that $\gR$ is the differential of the equality constraints of (\ref{SDPLR}) and $G'(R)[H] = \(\gR(H),H\)$. Let 
$$\Gamma_R:=\left\{ H\in \R^{n\times r}:\ U_1^\top H V_1+V_1^\top H^\top U_1=0\right\}.$$ 
The following definition is important for our later analysis.

\begin{defi}\label{SRE}
$R\in \R^{n\times r}$ such that $G(R)\in \K$ is called a regular point for (\ref{Oct_25_1}) if $\gR\(\Gamma_R\)=\R^{m-1}\times \R^r.$
\end{defi}

The above definition of a regular point is equivalent to the constraint nondegeneracy condition in Definition 2.1 of \cite{Sensitivity},
which states that
\begin{eqnarray}\label{Oct_31_1}
  G'(R) {\cal X} + {\rm lin}(\T_\K(G(R)) = {\cal Y}
\end{eqnarray}
where ${\cal X} =  \R^{n\times r}$,
${\cal Y} = \R^{m-1}\times \R^r\times \R^{n\times r}$, 
$\K = \{0\}^{m-1}\times \{0\}^r\times \Omega_\alpha^2$, and 
${\rm lin}(\T_\K(G(R))$ is the largest linear subspace contained in
$\T_\K(G(R)).$ In particular
$$
{\rm lin}(\T_\K(G(R)) = \{0\}^{m-1}\times \{0\}^r \times {\rm lin}(\T_{\Omega^2_\alpha}(R)).
$$
It is easy to show that the constraint nondegeneracy condition is equivalent to
$$
 \gR({\rm lin}(\T_{\Omega^2_\alpha}(R))) = \R^{m-1}\times \R^r.
$$
Since
$$
{\rm lin}(\T_{\Omega^2_\alpha}(R))
 = \{ Y\in\R^{n\times r} \mid U_1^\top Y V_1 + V_1^\top Y^\top U_1 = 0\}=\Gamma_R,
$$
we have shown the equivalence between a regular point and the 
constraint nondegeneracy condition.

\begin{rem}
Constraint nondegeneracy condition is also defined in \cite{Perturbation} Remark 4.72. They additionally require that the set $\K$ is $\mathcal{C}^1$-cone reducible (see Definition 3.135 of \cite{Perturbation}). One can prove that $\K=\{0\}^{m-1}\times \{0\}^r\times \Omega_\alpha^2$ is $\mathcal{C}^\infty$-cone reducible using a similar approach as in Example 3.140 of \cite{Perturbation}. We omit the proof here since we use the more convenient definition from \cite{Sensitivity}.
\end{rem}

\medskip
From (\ref{RobinCQ2}) and (\ref{Oct_31_1}), we have the following lemma.

\begin{lem}\label{RCQ}
If $R\in \R^{n\times r}$ such that $G(R)\in \K$ is regular, then it satisfies Robinson's constraints qualification for (\ref{Oct_25_1}). \end{lem}

The following lemma gives a sufficient 
condition for a {second order stationary point} of (\ref{SDPLR}) to be a global optimal solution.
\begin{lem}\label{local}
Suppose $R\in \R^{n\times r}$ such that $G(R)\in \K$ is a regular rank deficient 
{second order stationary point} of (\ref{SDPLR}). Then $RR^\top$ is a global optimal solution of (\ref{SDP1}) and so $R$ is a global optimal solution of (\ref{SDPLR}).
\end{lem}
\begin{proof}
First, from definition of $G$, we have 
$$ G'(R)[H]=\(\gR(H),H\)=\Big(\A\Big(\frac{RH^\top+HR^\top}{2}\Big),H^\top e,H\Big),
\ G''(R)[H,H]=\big( \A(HH^\top),0,0\big).
$$
Since $R$ is a regular {second order stationary point} of (\ref{SDPLR}), by Lemma~\ref{RCQ}, it satisfies Robinson's constraints qualification. This means that it satisfies first and second order necessary KKT conditions (See \cite{Perturbation} (3.16) and  (3.99)), i.e. there exists $(y,z)\in \R^{m-1}\times \R^r$ such that
\begin{equation}\label{KKT1}
G(R)\in \K,\ \(C-\A^*(y)\)R-ez^\top\in -\N_{\Omega_\alpha^2}(R);
\end{equation}
$$\forall\ d,w\in \R^{n\times r}\ \mbox{s.t}
\ G'(R)d=\(\gR(d),d\)\in \T_\K\( G(R)\)\ {\rm and}$$ 
$$G'(R)w+G''(R)\(d,d\)\in \T^{i,2}_{\K}\( G(R),G'(R)d\),\notag
$$
\begin{equation}\label{KKT2}
{\rm if}\ \<CR,d\>=0,\ {\rm then}\ \<CR,w\>+\<Cd,d\>\geq 0.
\end{equation}
Here $\T_\K^{i,2}(G(R),G'(R)d)$ is the inner second order tangent set defined in (\ref{istan}) as follows:
$$ \T_\K^{i,2}(G(R),G'(R)d)=\left\{ v\in \R^{n\times r}:\ {\rm dist} \Big(G(R)+tG'(R)d+\frac{t^2}{2}v,\K\Big)=o(t^2),\ t\geq 0\right\}.$$

From (\ref{KKT1}), we know that $X=RR^\top$ is feasible for (\ref{SDP1}) and there exists $Z\in \S^{r_1}$ such that $Z\preceq 0$ and 

\begin{equation} \label{Oct_27_1}
\(C-\A^*(y)\)R-ez^\top=U_1 Z V_1^\top\footnote{If $r_1$=0, then $U_1ZV_1^\top$=0.}.
\end{equation}

Our goal next is to show that $$ J\( C-\A^*y\)J\in -\N_{\Omega_\alpha^1}\(X\).$$

Multiplying $J$ to the left-hand side of (\ref{Oct_27_1}) and noting that $J(U_1,U_2,R)=(U_1,U_2,R),$ we get 
\begin{equation}\label{Sep15_2}
J\(C-\A^*(y)\)JR=U_1ZV_1^\top. 
\end{equation} 
Thus we have 
\begin{eqnarray*} 
\begin{pmatrix}U_1^\top\\U_2^\top\\U_3^\top\end{pmatrix}J\(C-\A^*(y)\)J\begin{pmatrix}U_1&U_2&U_3\end{pmatrix}\begin{pmatrix}
\sqrt{\Lambda_1}&&\\&\sqrt{\Lambda_2}&\\&&{\bf 0}\end{pmatrix}=\begin{pmatrix}Z&0&0\\0&0&0\\0&0&0\end{pmatrix},
\end{eqnarray*}
which implies that
\begin{equation}\label{Sep15_1}
J\(C-\A^*(y)\)J=\begin{pmatrix}U_1&U_2&U_3\end{pmatrix}\begin{pmatrix} Z\sqrt{\alpha}^{-1}&0&0\\0&0&0\\0&0&W\end{pmatrix}\begin{pmatrix}U_1^\top\\U_2^\top\\U_3^\top\end{pmatrix},
\end{equation}
for some $W\in \S^{n-r_1-r_2}.$ For any $S\in \R^{(n-r_1-r_2)\times (r-r_1-r_2)}$, define $d_S=JU_3SV_3^\top\in \R^{n\times r}.$ Note that we can define $d_S$ because $R$ is rank deficient and so $r-r_1-r_2>0$. Then $d_S^\top e=0$, and $Rd_S^\top=RV_3S^\top U_3^\top J=0$ because $RV_3=0.$ So $\gR(d_S)=0.$ We also have $$R+td_S=J
\begin{pmatrix}
U_1&U_2&U_3
\end{pmatrix}
\begin{pmatrix}
\sqrt{\Lambda_1}&&\\&\sqrt{\Lambda_2}&\\&&tS
\end{pmatrix}
\begin{pmatrix}
V_1^\top\\V_2^\top\\V_3^\top
\end{pmatrix},$$ 
which means that for $t$ sufficiently small $\|R+td_S\|_2\leq \sqrt{\alpha}$ i.e. $R+td_S\in \Omega_\alpha^2.$ Therefore, 
$$G'(R)d_S\in \T_\K\(G(R)\).$$ 
From $\gR\(\Gamma_R\)=\R^{m-1}\times \R^r,$ we may choose $w_S\in \Gamma_R$ such that 
$$\gR\(w_S\)=\(-\A\big(d_Sd_S^\top\big),0\).$$ 
Thus $\A\( \frac{Rw_S^\top+w_S R^\top}{2}+d_Sd_S^\top\)=0,$ and $w_S^\top e = 0$.

Next we show that 
\begin{equation} \label{Sep27_1}
\|R+td_S+\mbox{$\frac{1}{2}$}t^2 w_S\|_2\leq \sqrt{\alpha}+o(t^2).
\end{equation}
For the case where
$\|R\|_2<\sqrt{\alpha},$ it is easy to see that  (\ref{Sep27_1}) holds true. 
Now suppose $\|R\|_2=\sqrt{\alpha}.$ Then we have 
\begin{align}
&\lambda_{\max}\Big(  \Big(R+td_S+\frac{t^2}{2}w_S\Big)\Big(R+td_S+\frac{t^2}{2}w_S\Big)^\top  \Big)\notag \\
&=\lambda_{\max}\Big(  RR^\top+tRd_S^\top+td_SR^\top
+t^2\Big(d_Sd_S^\top+\frac{w_SR^\top+Rw_S^\top}{2}\Big) \Big)+O(t^3)\notag \\
&=\lambda_{\max}\Big(  RR^\top+t^2\Big(d_Sd_S^\top+\frac{w_SR^\top+Rw_S^\top}{2}\Big) \Big)+O(t^3)\notag \\
&=\lambda_{\max}\big(RR^\top\big)+t^2\lambda_{\max}\Big( U_1^\top\Big(d_Sd_S^\top+\frac{w_SR^\top+Rw_S^\top}{2}\Big)U_1\Big)+o(t^2)=\alpha+o(t^2),\notag 
\end{align}
where the first equality comes from the Lipschitz continuity of $\lambda_{\max},$ the second equality comes from $Rd_S^\top=0$,
 the third equality comes from the directional derivative of $\lambda_{\max},$ the fourth equality comes from $U_1^\top d_S=0$ and $U_1^\top w_S R^\top U_1+U_1^\top R w_S^\top U_1=0,$ where the latter
holds because $w_S\in \Gamma_R.$ This completes the proof of (\ref{Sep27_1}). \\

Now for $t>0$, we have 
\begin{align}
&\dist\(R+td_S+\frac{1}{2}t^2 w_S,\Omega_\alpha^2\)=o(t^2)
%\|R+td_S+ \frac{1}{2}t^2w_S\|_2\leq \sqrt{\alpha}+o(t^2)
\notag \\
&\(R+td_S+\frac{1}{2}t^2w_S\)^\top e=0 \notag \\
&\A\(\Big(R+td_S+\frac{1}{2}t^2w_S\Big)\Big(R+td_S+\frac{1}{2}t^2w_S\Big)^\top\)
-b\notag \\
&=t^2\A\( d_Sd_S^\top+\(Rw_S^\top+w_SR^\top\)/2\)+O(t^3)=O(t^3),\notag
\end{align}
where the first inequality comes from \eqref{Sep27_1} and the rest 
come from $\gR(w_S)=\( -\A\(d_Sd_S^\top\),0\).$ 
The above results together imply that
$$G'(R)w_S+G''(R)(d_S,d_S)\in \T^{i,2}_\K\(G(R),G'(R)d_S\).$$
 Since $\<CR,d_S\>=\<C,R d_S^\top\>=0$, from (\ref{KKT2}), we have 
$$ \<CR,w_S\>+\<Cd_S,d_S\>\geq 0.$$
From $\A\( d_Sd_S^\top+\(Rw_S^\top+w_SR^\top\)/2\)=0$, we have $\<\A^*(y),d_Sd_S^\top\>+\<\A^*(y)R,w_S^\top\>=0.$ 
Thus we have 
$$ \<\(C-\A^*(y)\)R,w_S\>+\inprod{\(C-\A^*(y)\)}{d_Sd_S^\top}\geq 0.$$ 

Since $Jw_S=w_S$, $JR=R$, $Jd_S=d_S$, it follows from the above inequality that 
$$ 
\<J\(C-\A^*(y)\)JR,w_S\>+\inprod{J\(C-\A^*(y)\)J}{d_Sd_S^\top}\geq 0.
$$ 
From (\ref{Sep15_2}) and $w_S\in \Gamma_R$, we have $\<J\(C-\A^*(y)\)JR,w_S\>=\<U_1ZV_1^\top,w_S\>= \inprod{Z}{U_1^\top w_S V_1}=0.$
 So 
 $$ 
 \<J\(C-\A^*(y)\)J,d_Sd_S^\top\>=\inprod{J\(C-\A^*(y)\)J}{U_3SS^\top U_3^\top}\geq 0,
 $$ 
 for any $S\in \R^{(n-r_1-r_2)\times(r-r_1-r_2)}.$ 
 This means that $W=U_3^\top J\(C-\A^*(y)\)J U_3\succeq 0.$ 
 Thus, from \eqref{Sep15_1},
  $$J\(C-\A^*(y)\)J \in -\N_{\Omega_\alpha^1}\(X\).$$  From Proposition~\ref{KKTcvx}, $X=RR^\top$ is an optimal solution for (\ref{SDP1}).
\end{proof}
%%%%%%%%%%%%%%%%%%%%%%%%%%%%%%%%%%%%%%%%%%%%%%%%%%%%%%%%
%%%%%%%%%%%%%%%%%%%%%%%%%%%%%%%%%%%%%%%%%%%%%%%%%%%%%%%%
\medskip
The following proposition is {inspired by Lemma 3 of \cite{Boumal1} and Theorem 2 in \cite{Boumal2}.} It tells that the rank deficient condition usually holds.
\begin{lem}\label{dcount}
Suppose $\beta:=\max\left\{ {\rm Tr(X)}:\ \A\(X\)=b,\ Xe=0,\ X\succeq 0\right\}<\infty$ and $r\geq \lceil\frac{\beta}{\alpha}\rceil+\sqrt{2m}.$ Then for all $C\in \S^n$ except for a zero measure set, any {first order stationary point} $R$ of (\ref{SDPLR}) that satisfies Robinson's constraints qualification is rank deficient.
\end{lem}

\begin{proof}
Let $C\in \S^n$ be such that (\ref{SDPLR}) has a full rank {first order stationary point} $R\in \R^{n\times r}$ satisfying Robinson's constraint qualification. It necessarily satisfies the 
first order KKT condition (\ref{KKT1}), which means that 
$$ \(C-\A^*(y)\)R\in ez^\top-\left\{ U_1ZV_1^\top:\ Z\succeq 0\right\}\subset \M_{\leq 1+r_1},$$ where $\M_{\leq 1+r_1}=\left\{A\in \S^n:\ \rr\(A\)\leq 1+r_1\right\}.$ Thus
 we have $$ \rr\(R\)+\rr\(C-\A^*(y)\)-n\leq \rr\( \(C-\A^*(y)\)R\)\leq r_1+1,$$ {where the first inequality follows from Sylvester’s rank inequality (see 10.5 of \cite{mcook}).} Note that $r_1\leq \lceil\frac{\beta}{\alpha}\rceil-1$ (for otherwise because $r>\lceil\frac{\beta}{\alpha}\rceil$ and $R$ is full rank, we would get the contradiction that ${\rm Tr}\(RR^\top\) >\lceil\frac{\beta}{\alpha}\rceil\alpha\geq\beta$).

Since $R$ has full rank, we have $$ \rr\(C-\A^*(y)\)\leq n+r_1+1-r\leq n-\sqrt{2m}.$$
Then $C\in \M_{\leq n-\sqrt{2m}}+\A^*\(\R^{m-1}\)$ whose dimension is bounded by $\frac{n(n+1)}{2}-\frac{\sqrt{2m}\(\sqrt{2m}+1\)}{2}+m-1<\frac{n(n+1)}{2}.$ This means 
that
$C\in\S^n$ is contained in a set of measure zero.
\end{proof}

The following theorem directly follows from Lemma~\ref{local} and Lemma~\ref{dcount}. 
\begin{theo}\label{lisg}
Suppose $\beta:=\max\left\{ {\rm Tr}\(X\):\ \A\(X\)=b,\ Xe=0,\ X\succeq 0\right\}<\infty$ and $r\geq \lceil\frac{\beta}{\alpha}\rceil+\sqrt{2m}.$ Then for all $C\in \S^n$ except for a zero measure set, any regular {second order stationary point} $R$ of (\ref{SDPLR}) is a global minimizer of $(\ref{SDPLR}).$ Moreover, if we can get $\Theta\in \N_{\Omega_2^\alpha}(R)$ in the KKT system (\ref{KKT1}), we have $\gR^*(y,z)=CR+\Theta.$ Because $\gR$ is an onto linear mapping, $(y,z)=\(\gR\, \gR^*\)^{-1}\gR\(CR+\Theta\).$ Thus, we can recover the dual variable for problem (\ref{SDP1}).
\end{theo}

\begin{rem} We note that the construction of the dual variables $(y,Z)$ from the primal variable
$R$ in 
Theorem \ref{lisg} is necessary since in our Riemannian based algorithm, 
the dual variables are not constructed explicitly. {We also note that Theorem~\ref{lisg} has the prerequisite that the output of an optimization algorithm is a regular point. Such a regularity assumption is commonly used in the optimization literature because without it, the 
 solution obtained by an algorithm may not even be a KKT solution.}
\end{rem}

%%%%%%%%%%%%%%%%%%%%%%%%
\subsection{Comparison between Theorem~\ref{lisg} and Theorem 4 in \cite{Cifu}}

{
We should mention that the SDP upper bound of (\ref{SDP1}) can also be handled by introducing a slack variable $Y$ and writing problem (\ref{SDP1}) equivalently as follows:
\begin{equation}\label{SDPC}
\min\left\{ \<C,X\>:\ \A(X)=b,\ \<X,ee^\top\>=0,\ X+Y=\alpha I,\ X,Y\in \S^n_+\right\}.
\end{equation}
In this case, 
problem (\ref{SDPC}) is a multi-block SDP problem. By using the partial BM factorization in section 4 of \cite{Cifu}, we get the following problem:
\begin{equation}\label{SDPCLR}
\min_{R\in \R^{n\times r}}\left\{ \<C,RR^\top\>:\ \A(RR^\top)=b,\ \<RR^\top,ee^\top\>=0,\ RR^\top+Y=\alpha I,\ Y\in \S^n_+\right\}.
\end{equation}

Theorem 4 in \cite{Cifu} says that if
\begin{equation}\label{Crank}
r(r+1)/2>m+n(n+1)/2-r(Y)\(r(Y)+1\)/2,
\end{equation}
then for a generic $C\in \S^n$, any 2-critical point $\(R,Y\)$ of (\ref{SDPCLR}) satisfies that $\(RR^\top,Y\)$ is global optimal for (\ref{SDPC}) (see section 2 of \cite{Cifu} for the definition of a 2-critical point). Here $r(Y)$ is a bound on the smallest rank of $Y$ for some feasible solution $\(X,Y\)$ of (\ref{SDPC}). From the definition of $\beta$ and $\alpha I\succeq X,$ we have that $r(Y)\geq n-\beta/\alpha.$ This implies the following rank condition of $r$:
\begin{equation}\label{Crank1}
r(r+1)/2>m+n(n+1)/2-\( n-\beta/\alpha\)\(n-\beta/\alpha+1\)/2.
\end{equation}

We assume that $\beta/\alpha\geq 1$; for otherwise the constraint $\alpha I\succeq X$ can be removed. If we have $1\leq \beta/\alpha=o(n)\footnote{This is true for SDP relaxations of graph equipartition and $k$-means clustering in (see \cite{Kmeans} (20)). Actually, for an SDP problem where the trace of $X$ is fixed as a constant $\beta$, the SDP upper bound $\alpha I\succeq X$ implies that $\rr\(X\)\geq \beta/\alpha.$ Thus, we have to assume $\beta/\alpha=o(n)$ if we want the BM approach to be efficient.}$ and $m=o(n^2)\footnote{This is often the case when we apply the BM approach. Otherwise there are too many constraints, and the SDP relaxation is not guaranteed to be low rank anymore.}$, then (\ref{Crank1}) gives the following rank bound
\begin{equation}\label{Crank2}
r>\sqrt{2m+(2+o(1))n\beta /\alpha}.
\end{equation}
In comparison, our rank condition is 
\begin{equation}\label{Orank}
r\geq \lceil\frac{\beta}{\alpha}\rceil+\sqrt{2m}=\sqrt{2m+o(1)n\beta/\alpha},
\end{equation}
which is strictly better than the condition in (\ref{Crank2}), especially when $n$ is sufficiently large. In particular, in graph equipartition with $k$ parts, (\ref{Crank2}) implies that $r>\sqrt{(1+o(1))2kn},$ while (\ref{Orank}) implies that $r\geq (1+o(1))\sqrt{2n}.$ The latter is clearly better than the former bound when $k$ is large. Moreover, if one considers the SDP relaxation of a $k$-means clustering problem (see \cite{Kmeans} (20)), then (\ref{Crank2}) implies that $r>\sqrt{(1+o(1))(k-1)n},$ while our condition gives $r\geq \sqrt{2}+k-1,$ 
which is a constant that is independent of $n$. In contrast, the former lower bound derived from (\ref{Crank2})
can be large when $n$ is much larger than $k$, which is usually the case for 
a clustering problem. Apart from the rank condition, our BM factorization (\ref{SDPLR}) only involves one low-rank matrix variable $R,$ while (\ref{SDPCLR}) also involves a large high rank positive semidefinite matrix variable $Y$\footnote{Note that $RR^\top +Y=\alpha I.$ If $r$ is small, then $Y$ must by high rank.}, 
which would lead to
expensive computation when one has to handle such a high-dimensional matrix variable
in any algorithm for solving  (\ref{SDPCLR}).
Moreover, we are not aware of an existing algorithm 
that can handle the problem  (\ref{SDPCLR}) efficiently.
}

%%%%%%%%%%%%%%%%%%%%%%%%%%%%%%%%%%%%%%%%%%%%%%%%%%%%%%%%%%%%%%
%%%%%%%%%%%%%%%%%%%%%%%%%%%%%%%%%%%%%%%%%%%%%%%%%%%%%%%%%%%%%%
\section{An algebraic variety}

In this section, we consider the case where the mapping $\A(\cdot)=\dd(\cdot)$. The constraints of (\ref{SDPLR}) come from two parts, one of them is the following set,
\begin{equation}\label{BCalg}
\B_{n,r}=\left\{ R\in \R^{n\times r}:\ \dd\big(RR^\top\big)=e,\ R^\top e=0\right\}.
\end{equation}
We remind the reader that $\dd(RR^\top)=e$ is equivalent to $\|e_i^\top R\|_2^2=1$ for all
 $i\in [n].$ {We call a point $R\in \B_{n,r}$ smooth (nonsingular) point if the linear independent constraint qualification (LICQ) holds at $R$. Otherwise we say that $R$ is nonsmooth (singular). Strictly speaking, LICQ is a sufficient but not necessary condition for a point to be smooth. However, we identify smoothness and LICQ because LICQ is convenient for us to use in optimization.}
 
Observe that $\B_{n,r}$ is an algebraic variety defined as the common zeros of $n+r$ polynomials (see \cite{AlgV,agtan}). Its tangent space {at a smooth point} $R\in \B_{n,r}$ is given by
\begin{equation}\label{BCtan}
T_{R}\B_{n,r}:= \left\{ H\in \R^{n\times r}:\  H^\top e=0,\ \dd(HR^\top)=0\right\}.
\end{equation}
\subsection{Local geometric properties of $\B_{n,r}$}
From now on, we let $E^n:=\left\{ v\in \{-1,1\}^n:\ e^\top v=0\right\}.$
Note that $\B_{n,r}$ is not always a smooth manifold and we have the following decomposition result.

\begin{prop}\label{decomp}
$\B_{n,r}=\B_{n,r}^0\sqcup \B_{n,r}^1$, {where $\sqcup$ denotes the disjoint union of the sets} $\B_{n,r}^0:=\left\{ vb^\top:\ v\in E^n,\ b\in {\rm S}^{r-1}\right\}$ and 
$\B^1_{n,r}:=\left\{R\in \B_{n,r}:\ \rr\left(R\right)\geq 2\right\}.$ Moreover, $\B_{n,r}^0$ contains the singular points of $\B_{n,r}$ and $\B_{n,r}^1$ contains the smooth points of $\B_{n,r}.$ {Note that when $n$ is odd, $\B^0_{n,r}=\emptyset.$}
\end{prop} 
\begin{proof}
The decomposition comes from the fact that any $R\in \B_{n,r}$ with rank equal to 1 has the form $R=vb^\top,$ for some $v\in E^n$ and $b\in {\rm S}^{r-1}$. For any $R\in \B_{n,r}$, define the linear mapping
 $\A_R:\R^{n\times r}:\ \rightarrow \R^r\times \R^n,$ such that for any $H\in \R^{n\times r}$, $\A_R(H)=\left(H^\top e,\dd(HR^\top)\right)$. Then for any $(\lambda,\mu)\in \R^r\times \R^n$, $\A_R^*\left( (\lambda,\mu)\right)=e\lambda^\top+\dd(\mu)R.$ It is easy to see that LICQ is equivalent to that $\A_R^*$ is an injection. For $R= vb^\top\in \B_{n,r}^0,$ we have that $\A_R^*\((-b,v)\)=0.$ So, $R$ is singular. For any $R\in \B_{n,r}^1$, we only have to prove that $\A_R^*$ is an injection. Suppose $\A_R^*\left((\lambda,\mu)\right)=e\lambda^\top+\dd(\mu)R=0$. If there exists $i\in [n]$ such that $\mu_i=0$, then $\lambda=0$ and $\dd(\mu)R=0.$  Because $\dd(RR^\top)=e$, we have $\mu=0.$ If for all $i\in [n]$, $\mu_i\neq 0$, then $R=-\dd(\mu)^{-1}ea^\top$, which is rank 1, and we get a contradiction. Therefore, $(\lambda,\mu)=0$ and $\A_R^*$ is injective. We have that $R$ is a smooth point of $\B_{n,r}$.
\end{proof}

%%%%%%%%%%%%%%%%%%%%%%%%%%%%%%%%Singular points%%%%%%%%%%%%%%%%%%%
{
Proposition~\ref{decomp} shows that the only singular points in $\B_{n,r}$ are rank-1 points. We should note that when using the Burer and Monteiro factorization, $R\in \R^{n\times r}$ is equivalent to $RQ$ for any orthogonal $Q\in \St(r,r).$ This is because $RR^\top=RQQ^\top R^\top.$ Thus, for any singular point $R=vb^\top$ such that $v\in E^n$ and $b\in {\rm S}^{r-1},$ there exists an orthogonal matrix $Q\in \St(r,r)$ such that $RQ=ve_1^\top.$ This implies that there are only finitely many classes of singular points. Moreover, the number of classes is $\frac{1}{2}\binom{n}{\frac{n}{2}}$ if $n$ is even and $0$ if $n$ is odd. For convenience, we only consider singular points of the form $R=ae_1^\top$ for some $a\in E^n.$

Now we move on to characterize the tangent cone and second order tangent set at a singular point of $\B_{n,r}$ in the following two propositions. We defer their proofs to Appendix C. 

\begin{prop}\label{pcone}
Suppose $r>1$ and $R = ae_1^\top\in \B_{n,r}$ is a singular point for some $a\in E^n$, then its tangent cone is given by:
\begin{equation}
\T^i_{\B_{n,r}}(R)=\T_{\B_{n,r}}(R)
= \widetilde{\T}:=\big\{ [0,H_1]:H_1\in \R^{n\times (r-1)},e^\top H_1=0,a^\top \dd\big(H_1H_1^\top\big)=0\big\}.
\label{sgtcone}
\end{equation}
\end{prop}

\begin{prop}\label{sectan}
Suppose $r>1.$ Let $R=ae_1^\top$ be a singular point for some $a\in E^n.$ Then for any $H=[0,H_1]\in \T^i_{\B_{n,r}}(R)=\T_{\B_{n,r}}(R).$ 
\begin{itemize}
\item[(i)] If $H_1\notin a\R^{1\times (r-1)},$ then
\begin{align}
&\T^{i,2}_{\B_{n,r}}(R,H)=\T^2_{\B_{n,r}}(R,H) =\widetilde{\T}^2(R,H)
\notag \\
& :=\Big\{ \left[-a\circ\dd\big(H_1H_1^\top\big),W_1\right]:\ a^\top \dd\big(H_1W^\top_1\big)=0, \,
e^\top W_1=0, \,W_1  \in \R^{n\times (r-1)}\Big\}.\notag
\end{align}
\item[(ii)] If $H_1=a\lambda^\top$ for some $\lambda\in \R^{r-1},$ then
\begin{align}
&\T^{i,2}_{\B_{n,r}}(R,H)=\T^2_{\B_{n,r}}(R,H)=\widetilde{\T}^2(R,H)
\notag \\
& :=\Big\{ \left[-a\circ\dd\big(H_1H_1^\top\big),W_1\right]:\ a^\top \dd\big(W_1W^\top_1\big)=0, \, 
e^\top W_1=0, \,W_1  \in \R^{n\times (r-1)}\Big\}.\notag
\end{align}
\end{itemize}
\end{prop}

Proposition~\ref{pcone} and \ref{sectan} together imply that $\B_{n,r}$ is second order directionally differentiable at every singular point (see definition 3.32 of \cite{Perturbation}). With these geometric properties, we can move on to study operations on $\B_{n,r},$ which are crucial for designing algorithms. We will first discuss some important operations on the smooth part of $\B_{n,r}$ such as projection and retraction in the next two sections, while 
the discussion on how to handle the singular points will be given later in Appendix B 
since we rarely encounter singular points in practical computations.
}

%%%%%%%%%%%%%%%%%%%%%%%%%%%%%%%%%%%%%%%%%%%%%%%%%%%%%%%%
\subsection{Projection onto tangent spaces {of smooth points on} $\B_{n,r}$}

For notational convenience, let $\P_{R}(C):=\P_{T_{R}\B_{n,r}}(C).$
\begin{theo}
For any smooth point $\barR\in \B_{n,r}$ and $C\in \R^{n\times r},$
\begin{equation}\label{BCproj}
\P_{\barR}(C)=J\left(C-\dd\left(\lambda_C^{\barR}\right)\barR\right),
\end{equation}
where 
\begin{equation}\label{lamb}
\lambda_C^{\barR}:=\Big(I+\barR\left( I-\barR^\top \barR/n\right)^{-1} \barR^\top/n\Big)\dd\big( JC\barR^\top\big).
\end{equation}
\end{theo} 
\begin{proof}
Suppose $\barR\in \B_{n,r}^1$, we have
\begin{equation}\label{BCproj1}
\P_{\barR}(C):=\arg\min_{H\in \R^{n\times r}}\left\{ \frac{1}{2}\|H-C\|^2:\ H^\top e=0,\ \dd(H\barR^\top)=0\right\}.
\end{equation}
Since, the feasible region in (\ref{BCproj1}) is convex and the objective function is strongly convex, it has a unique optimal solution, say $H.$ From LICQ property of $R$, $H$ also satisfies LICQ. Then $H$ must be a KKT solution, that is, there exists $\lambda_1\in \R^r$, $\lambda_2\in \R^n$ such that
\begin{align}
H-C+e\lambda_1^\top+\dd(\lambda_2)\barR=0\label{eq1}\\
H^\top e = 0\label{eq2}\\
\dd(H\barR^\top)=0\label{eq3}
\end{align}

Multiplying $J$ to the left hand side of (\ref{eq1}), and noting that $e^\top H=0$, we have 
$$ H=J(C-\dd(\lambda_2)\barR).$$ 
Plug this into (\ref{eq3}), we have 
\begin{equation}\label{eq4} 
\dd\big( J\dd(\lambda_2)\barR\barR^\top\big)=\dd\big(JC\barR^\top\big).
\end{equation} 
Apply the following formula to (\ref{eq4})
\begin{equation}\label{interesting}
\dd\left( A\dd(\lambda)B\right)=\big(A\circ B^\top\big) \lambda,
\end{equation}
we have 

\begin{equation}\label{eq5}
 \left(I-\barR\barR^\top/n\right)\lambda_2=\dd\big( JC\barR^\top\big)
\end{equation}
Since $\barR\barR^\top/n\succeq 0$ and ${\rm Tr}(\barR\barR^\top/n)=
{\rm Tr}(I/n)=1,$ we have $\|\barR\barR^\top/n\|_2\leq 1$ and the equality holds if and only if $\barR$ is rank 1. Because $\rr(\barR)\geq 2$, we must have $I-\barR\barR^\top/n\succ 0.$  From the Sherman-Morrison-Woodbury formula, we have 
\begin{equation}\label{lam}
 \lambda_2=\Big(I+\barR\left( I-\barR^\top \barR/n\right)^{-1}\barR^\top/n\Big)\dd\big( JC\barR^\top\big),
 \end{equation}
which is exactly $\lambda_C^{\barR}.$
\end{proof}

%%%%%%%%%%%%%%%%%
\subsection{Retraction around smooth points of $\B_{n,r}$}
\label{subsec-retraction}

For $\barR\in \B_{n,r}$ and $H\in T_{\barR}\B_{n,r}$, define 
\begin{equation}\label{BCretrac}
\Re_{\barR}(H):=\P_{\B_{n,r}}(\barR+H)\in \arg\min\left\{ \|R-(\barR+H)\|^2/2:\ R\in \B_{n,r}\right\}.
\end{equation}
Since $\B_{n,r}$ is compact, 
the optimization problem (\ref{BCretrac}) always has an optimal solution, but it may not be unique as $\B_{n,r}$ is non-convex. Now, we focus on the following optimization problem. Given $Y\in \R^{n\times r},$
\begin{equation}\label{BCopt}
\min\Big\{ \frac{1}{2}\|R-Y\|^2:\ R\in \B_{n,r}\Big\}.
\end{equation}

For $Y_1,Y_2\in \R^{n\times r}$, we say $Y_1\sim Y_2$ if there exists $b\in \R^r$ such that $Y_1=Y_2+eb^\top.$ It is easy to check that $\sim$ is an equivalence relation, which partitions $\R^{n\times r}$ into equivalence classes $[Y]:=\left\{Y_1\in \R^{n\times r}:\ Y_1\sim Y\right\}.$

\begin{lem}\label{equi}
For $Y\in \R^{n\times r}$, and $Y_1\in [Y],$ 
\begin{equation}\label{equi1}
\arg\min\left\{ \|R-Y\|^2/2:\ R\in \B_{n,r}\right\}=\arg\min\left\{ \|R-Y_1\|^2/2:\ R\in \B_{n,r}\right\}.
\end{equation}
\end{lem}
\begin{proof}
Since $Y_1\in [Y]$, $Y_1=Y+eb^\top$ for some $b\in \R^r$. Because for any $R\in \B_{n,r}$, $R^\top e=0$, the two objective functions in (\ref{equi1}) only differs by a constant. Thus, the optimal solutions set are the same.
\end{proof}

Now, we move on to show that problem (\ref{BCopt}) is related to the well known geometric median problem. This is the key point for solving (\ref{BCopt}) efficiently. For $Y\in \R^{n\times r}$ such that every row is non-zero, define the mapping 
$\O:\R^{n\times r}\rightarrow \R^{n\times r}$ such that $\O(Y):=Y\(\DD(YY^\top)\)^{-\frac{1}{2}},$ which is derived from $Y$ by normalizing each row of $Y$ to have unit $2$-norm. For any $Y\in \R^{n\times r},$ define $$\F(Y):=\{\O(Y_1):\, Y_1\in [Y]\}\,\cap\, \B_{n,r}.$$

\begin{lem}\label{easy}
For any $Y\in \R^{n\times r}$, either $\F(Y)=\emptyset$ or 
$ \F(Y)=\arg\min\left\{ \frac{1}{2}\|R-Y\|^2:\, R\in \B_{n,r}\right\}$ and $|\F(Y)|=1.$
\end{lem}
\begin{proof}
Suppose $\F(Y)\neq \emptyset$. Then there exists $Y_1\in [Y]$ such that $\O(Y_1)\in \B_{n,r}.$ Because $\O(Y_1)=\P_{\OB_{n,r}}(Y_1)$ and $\B_{n,r}\subset \OB_{n,r},$ 
we have $\{\O(Y_1)\}=\arg\min\left\{ \frac{1}{2}\|R-Y_1\|^2 :\, R\in \B_{n,r}\right\}.$ From Lemma~\ref{equi}, $\{\O(Y_1)\}=\arg\min\left\{ \frac{1}{2}\|R-Y\|^2:\, R\in \B_{n,r}\right\}.$ From the arbitrariness of $Y_1$, we have $\F(Y)\subset \arg\min\left\{\frac{1}{2}\|R-Y\|^2:\, R\in \B_{n,r}\right\}$. Thus,
$$\O(Y_1)\in \F(Y)\subset \arg\min\Big\{\frac{1}{2}\|R-Y\|^2:\, R\in \B_{n,r}\Big\}=\{\O(Y_1)\},$$
and the proof is completed.
\end{proof}

From Lemma ~\ref{easy}, we know that if $\F(Y)\neq \emptyset$, then $\F(Y)= \P_{\B_{n,r}}(Y).$ The next proposition gives a sufficient and necessary condition for $\F(Y)\neq \emptyset.$

\begin{prop}\label{sl}
For any $Y\in \R^{n\times r}$, $\F(Y)\neq \emptyset$ if and only if the following geometric median problem,
\begin{equation}\label{geo}
\min\Big\{ f(b):=\sum_{i=1}^n\|y_i-b\|_2:\ b\in \R^r\Big\},
\end{equation}
has an optimal solution $b_0\notin \{y_i : i\in [n]\}.$ Here, $y_i^\top$ is the $i$th row of $Y.$
\end{prop}
\begin{proof}
Note that (\ref{geo}) is a convex optimization problem, it has an optimal solution 
$b_0\notin \{y_i :i\in [n]\}$ if and only if $$ \nabla f(b)\Big|_{b=b_0}=\sum_{i=1}^n\frac{b_0-y_i}{\|b_0-y_i\|}=0.$$ The latter condition is equivalent to $\F(Y)\neq \emptyset.$
\end{proof}

From Proposition~\ref{sl} and Lemma~\ref{easy}, we may solve (\ref{BCopt}) by solving (\ref{geo}). Thus, a nonconvex problem is reduced to a well studied convex problem. There are many algorithms to solve (\ref{geo}) with guaranteed convergence, one simple method is the Weiszfeld algorithm (\cite{Weiszf}), which performs the following iteration:
\begin{equation}\label{Weiszfeld}
b_{i+1}=\left( \sum_{j=1}^n\frac{y_j}{\|b_i-y_j\|}\right)\Big/\left( \sum_{j=1}^n\frac{1}{\|b_i-y_j\|}\right).
\end{equation}
However, the condition $b_0\notin \{y_i : i\in [n]\}$ in Proposition~\ref{sl} may not always hold. Consider $Y\in \R^{3\times 2}$, whose rows are not collinear (not on the same line), it is well known that the geometric median of $y_1,y_2,y_3$ belongs to $\{y_i:\ i\in [3]\}$ if and only if the triangle $y_1y_2y_3$ has an angle $\geq 120^\circ.$ This means that the condition $b_0\notin \{y_i :i\in [n]\}$ may not hold with probability 1.  
Moreover, the following example show that one cannot find a neighbourhood of  $\B_{n,r}$ such that every $Y$ in it satisfies $\F(Y)\neq \emptyset.$

\begin{exmp}\label{exm}
Consider $\R^{4\times 2}.$ For any $\epsilon>0$, define $Y_\epsilon\in \R^{4\times 2}$ such that $$Y_\epsilon:=\begin{pmatrix} 0&0&0&-\frac{\epsilon}{2}\\1&1-\frac{\epsilon}{2}&-1&-1+\frac{\epsilon}{2}
\end{pmatrix}^\top.$$ We have $\dist\(Y_\epsilon,\B_{n,r}\)<\epsilon$. Because the first three rows of $Y_\epsilon$ are on the same line, the only geometric median of its rows vector is $(0,1-\frac{\epsilon}{2})^\top$, so $\F\(Y_\epsilon\)= \emptyset.$ However, as $\epsilon\rightarrow 0$, $Y_\epsilon$ can be arbitrarily close to $\B_{4,2}.$
\end{exmp}

The following theorem is a sufficient condition for the geometric median method to be useful.

\begin{theo}\label{analysis}
Suppose $Y\in \R^{n,r}$ satisfies $\|Y\|_2<(1-\beta)\sqrt{n}$ for some $0<\beta<1,$ and $\dist\(Y,\B_{n,r}\)<\frac{\beta \sqrt{n}}{\sqrt{n}+1}<1.$ Then $\F(Y)\neq \emptyset.$
\end{theo}

\begin{proof}
Let $R$ be an optimal solution of (\ref{BCopt}). Then $\|R-Y\|_F=\dist\(Y,\B_{n,r}\).$ 
Let $\|R-Y\|_F=\alpha \sqrt{n}.$ Then we have $\alpha<\frac{\beta}{\sqrt{n}+1}$ and hence
$$
 \beta-\alpha>\alpha\sqrt{n}.
 $$ 
 Moreover $\|R\|_2\leq \|Y\|_2+\|Y-R\|_F<\(1-\beta+\alpha\)\sqrt{n}.$ 
Since any rank-one solution 
 of \eqref{BCopt} must have spectral norm equal to $\sqrt{n}$, the last inequality 
  implies that $\rr\(R\)\geq 2$ 
 and $$\Big\|\(I-RR^\top/n\)^{-1}\Big\|_2\leq \frac{1}{1-\|RR^\top/n\|_2}<\frac{1}{1-\(1-\beta+\alpha\)^2}. $$

By Proposition~\ref{decomp}, $R$ is smooth. So it is a KKT solution of (\ref{BCopt}). There exist $\mu\in \R^r$, $\lambda\in \R^n$ such that the following system holds:
\begin{align}
R-Y+e\mu^\top+\dd(\lambda)R=0 \label{Sep8_1}\\
R^\top e=0 \label{Sep8_2}\\
\dd(RR^\top)=e \label{Sep8_3}
\end{align}
Multiplying $J$ to the left hand side of (\ref{Sep8_1}) and using (\ref{Sep8_2}), we have $$R=J\(Y-\dd(\lambda)R\).$$ Then
$$\dd(RR^\top)=\dd\big( J\(Y-\dd(\lambda)R\)R^\top\big).$$ 
Using (\ref{Sep8_3}) and (\ref{interesting}), we have 
$$ \(I-RR^\top/n\)\lambda =\dd\big(JYR^\top\big)-e.$$
Thus we have that
$$ \lambda =\(I-RR^\top/n\)^{-1}\big(\dd(JYR^\top)-e\big)
=\(I-RR^\top/n\)^{-1}\dd\big(J(Y-R)R^\top\big).
$$ 
Hence
\begin{align}
&\|\lambda\|_2\leq \Big\|\(I-RR^\top/n\)^{-1}\Big\|_2\,
\big\| \dd\big(J(Y-R)R^\top\big)\big\|_2\leq \frac{\big\| \dd\(J(Y-R)R^\top\)\big\|_2}{1-\(1-\beta+\alpha\)^2}.\label{Sep9_1}
\end{align}

Note that we also have
\begin{align}
&\Big\| \dd\big(J(Y-R)R^\top\big)\Big\|_2\leq \big\| \dd\big((Y-R)R^\top\big)\big\|_2
+\big\| R(Y-R)^\top e/n\big\|_2\notag \\[5pt]
&\leq \|Y-R\|_F+\|R\|_2\|Y-R\|_2\|e\|_2/n<\(2-\beta+\alpha\)\alpha\sqrt{n}\label{Sep9_2},
\end{align}
where we used {Cauchy-Schwarz} inequality and $\dd\(RR^\top\)=e$  in the second inequality of (\ref{Sep9_2}). Plug (\ref{Sep9_2}) into (\ref{Sep9_1}), we have 
\begin{equation*}\label{Sep9_3}
\|\lambda\|_2\leq \frac{\alpha\sqrt{n}}{\beta-\alpha}<1.
\end{equation*}
Now, from (\ref{Sep8_1}), we have $\(I+\dd\(\lambda\)\)R=Y-e\mu^\top.$ Since $\|\lambda\|_2<1$, we have $\(I+\dd\(\lambda\)\)\succ 0$. Hence $R=\O(Y-e\mu^\top)$ and so $\F(Y)\neq \emptyset.$
\end{proof}

\medskip
\begin{rem}\label{remretrac}
One may get the impression that the result in Theorem~\ref{analysis} contradicts Example~\ref{exm}. Intuitively speaking, Theorem~\ref{analysis} requires that a point's distance to $\B_{n,r}$ is smaller than its distance to a rank-1 matrix up to some constant, which is not equivalent to saying that the matrix is sufficiently close to $\B_{n,r}.$
Theorem~\ref{analysis} is useful for designing algorithms for
solving the $k-$equipartition problem GEP1. Consider the problem in (\ref{BCretrac}). In practice, $R$ is usually far from being rank-1. Moreover, for $k\geq 3$, the $k-$equipartition problem has the constraints $\|R\|_2\leq \sqrt{\frac{n}{k-1}}.$ If we choose $\beta\in \big(0,\frac{\sqrt{n}+1}{\sqrt{n}+2}\big(1-\frac{1}{\sqrt{k-1}}\big)\big)$ in Theorem~\ref{analysis}
and $H$ to be sufficiently small such that $\dist\(R+H,\B_{n,r}\)\leq \|H\|_F<\frac{\beta\sqrt{n}}{\sqrt{n}+1},$ then we have 
$$
\|Y\|_2=\|R+H\|_2\leq \|R\|_2+\|H\|_2=\sqrt{\frac{n}{k-1}}+\frac{\beta\sqrt{n}}{\sqrt{n}+1}<(1-\beta)\sqrt{n}.
$$ 
This means that, if $\rr\(R\)\geq 2$, we can first compute the retraction 
$\Re_{R}(H)$ by solving (\ref{geo}), with $Y=R+H$. If $\F(Y)\neq \emptyset$, then we have computed the retraction successfully, otherwise, {we can perform} a backtracking step:
$H\leftarrow H/\sigma$ for some $\sigma>1$. Theorem~\ref{analysis} tells us that after a finite number of backtracking steps, $\F(Y)\neq \emptyset$ and the retraction computation
will be successful. Note that the condition $\rr\(R\)\geq 2$ is usually satisfied in our experiments. Moreover, for $k\geq 3$, any feasible $R$ must satisfy $\rr(R) \geq 2$ 
since we have
 $\|R\|_2\leq \sqrt{\frac{n}{k-1}}<\sqrt{n}$. The last inequality rules out 
the possibility for $R$ to have rank one since in that case its spectral norm must be 
equal to $\sqrt{n}$.

Up to now, we are able to compute the retraction of a point only if it is ``smooth". It is an interesting mathematical problem to compute the retraction in the general case. Suppose $Y\in \R^{n\times r}=vb^\top$ such that $b\in {\rm S}^{r-1}$ is a rank-1 matrix. Without loss of generality, assume that
the entries of $v$ are arranged in an ascending order: $v_1\leq v_2\leq \ldots,\leq v_n$. If $n$ is an even number, then $\P_{\B_{n,r}}\(Y\)= wb^\top,$ where $w_1=\ldots=w_{n/2}=-1$, $w_{n/2+1}=\ldots=w_{n}=1.$ This can be proved by first assuming $v_1< v_2< \ldots<v_n$, in which case we can apply Proposition~\ref{sl}. Then use the continuity of $\P_{\B_{n,r}}\(Y\)$ to prove the case for $v_1\leq v_2\leq \ldots\leq v_n.$ 
\end{rem}

%%%%%%%%%%%%%%%%%%%%%%%%%%%%Algorithm%%%%%%%%%%%%%%%%%%%%%%%%%%%%
\section{Algorithms}
In this section, we consider algorithms to solve \eqref{BCsmo}, with $\A(\cdot)=\dd(\cdot)$ and $\alpha=\frac{n}{k-1}.$ We will discuss the algorithms for minimum bisection and graph multi-equipartition seperately.

\subsection{Minimum bisection}
{
In this subsection, we consider the case $k=2$ and hence $\alpha =n$. In this case, 
since ${\rm diag}(RR^\top) = e$ implies that $\| R \|_F = \sqrt{n}$, the constraint 
$\|R\|_2\leq \sqrt{\alpha}$ is redundant and it can be removed. 
%The reason we remove the upper bound is that it only appears in graph equipartition with more than 2 parts. In the latter case, the spectral upper bound already implies that every feasible point $R$ is smooth on $\B_{n,r}$ (see Remark~\ref{remretrac}). 
We rewrite the minimum bisection problems as follows:
\begin{equation}\label{BCSDP}
\min\Big\{ \<C,X\>:\ \dd\(X\)=e,\ Xe=0,\ \<X,ee^\top\>=0,\ X\in \S^n_+\Big\},
\end{equation}
\begin{equation}\label{BCsmo}
\min\Big\{ f(R):=\frac{1}{2}\< C,RR^\top\>:\ R\in \B_{n,r}\Big\}.
\end{equation}

From section 3, we know that $\B_{n,r}$ can be viewed as an affine variety with finitely many classes of singular points. In particular, for any $\delta\in (0,1),$ $\B_{n,r}$ can be separated as $\B_{n,r}^{\delta-}$ and $\B_{n,r}^{\delta+}$ such that 
\begin{equation}\label{BC-}
\B_{n,r}^{\delta-}:=\left\{R\in \B_{n,r}:\ \|R\|_2< \sqrt{1-\delta}\sqrt{n}\right\},
\end{equation}
\begin{equation}\label{BC+}
\B_{n,r}^{\delta+}:=\Big\{R\in \B_{n,r}:\ \|R\|_2\geq \sqrt{1-\delta}\sqrt{n}\Big\}.
\end{equation}
For any $\delta\in (0,1),$ all singular points are contained in $\B_{n,r}^{\delta+}.$ If there exists some constant $\delta\in (0,1)$ such that all iterations are in $\B_{n,r}^{\delta-},$ then we can use any kind of Riemannian optimization algorithm to solve (\ref{BCsmo}). Actually, this is often the case in practice. Only in very rare situation  will one encounter a singular point. For example, if the optimal solution of (\ref{BCsmo}) is rank-1, then the iterations will approach a singular point. Note that our escaping strategy  in Appendix B only works for an exact singular point. Thus, we will use a rounding procedure that will round any ``nearly singular point" to a singular point and then apply Theorem~\ref{esctheo} to handle it. For $\delta<\frac{1}{2},$ define the
 function $\Ron:\B_{n,r}^{\delta+}\rightarrow \R^{n\times r}$ such that for any $R\in \B_{n,r}^{\delta+},$ $\Ron(R):=\sgn(u)e_1^\top,$ where $u\in \R^n$ is the singular vector (with 
 its first nonzero component being positive)
 that corresponds to the largest singular value of $R$ and $\sgn$ is the sign function. Note that since $\|R\|_F^2=n$ for $R\in \B_{n,r},$ so when $\delta<\frac{1}{2},$ 
 the leading singular value has multiplicity one and $u$ must be unique.
 Hence $\Ron(.)$ is well-defined. We now state our algorithm. 

\begin{description}
\item [Algorithm 1:] Choose $R_0\in \B_{n,r}$ and $\delta_0\in (0,\frac{1}{2}).$ Set $k=0.$
\item[Step 1.] Set $R_k$ as the initial point, $f_k:=f(R_k).$ Use any Riemannian optimization method to solve (\ref{BCsmo}) such that every iteration $R\in \B_{n,r}$ satisfies $f(R)\leq f_k$. If for some iteration, we obtain a point $\widehat{R}_k\in \B_{n,r}^{\delta_k+},$ goto Step 2.
\item[Step 2.]
{\bf Case 1.} $\Ron\big(\widehat{R}_k\big)\notin \B_{n,r},$ set $R=\widehat{R}_k,$ goto Step 3.\\
{\bf Case 2.} $\Ron\big(\widehat{R}_k\big)\in \B_{n,r}$ and is optimal, stop.\\
{\bf Case 3.} $\Ron\big(\widehat{R}_k\big)\in \B_{n,r}$ and is non-optimal: Use the escaping strategy in Appendix B to find another smooth point $R^+\in \B_{n,r}$ such that $f(R^+)<f(\Ron(\widehat{R}_k))-{\epsilon\big(\Ron(\widehat{R}_k)\big)},$ where 
{$\epsilon\big(\Ron(\widehat{R}_k)\big)>0$} is only related to 
{$\Ron(\widehat{R}_k)$}. If $f(R^+)<f(\widehat{R}_k)$, set $R\leftarrow R^+;$ Otherwise, set $R\leftarrow \widehat{R}_k.$\footnote{In this case, $R$ is also a smooth point. This is because otherwise $f(R^+)<f(\Ron(\widehat{R}_k))=f(\widehat{R}_k)$ and this case won't happen.} Goto Step 3.\\
\item[Step 3.] Set $\delta_{k+1}\leftarrow \delta_k/2$, $R_{k+1}\leftarrow R$,\ $k\leftarrow k+1,$ goto Step 1.\\
\end{description}

%%%%%%%%%%%%%%%%%%%%%%%%%%%%%%%BB step%%%%%%%%%%%%%%%%%%%%%%%%%%%%
%%%%%%%%%%%%%%%%%%%%%%%%%%%%%%%%%%%%%%%%%%%%%%%%%%%%%%%%%%%%%%%
In Step 1, many Riemannian optimization methods with linesearch can satisfy $f(R)\leq f_k$ for any iteration $R.$ In practice, we will use a Riemannian gradient descent method with Barzilai-Borwein (BB) step and nonmonotone linesearch (see \cite{BB,adaptive,Intromani,RieBB,BB2}) in Step 1 of Algorithm 1. For $R\in \B_{n,r}$, the Riemannian gradient of a smooth function $f$ is given by $\g f(R):=\P_R\(\nabla f(R)\).$ 
\begin{description}
\item[Riemannian gradient BB method:] 
Given $R_0\in \B_{n,r},\alpha_0$, integer $M\geq 0$, $\gamma\in (0,1)$, $\beta>0,$ $\sigma>0$, $\epsilon>0$. Set $k=0.$
\item{[Step 1]} Compute the  Riemannian gradient $g_k=\g f(R_k)$. If $\|g_k\|=0$, stop.
\item{[Step 2]} If $\alpha_k\leq \epsilon$ or $\alpha_k\geq 1/\epsilon,$ then set $\alpha_k=\beta$
\item{[Step 3]} Set $\tau=1/\alpha_k.$
\item{[Step 4]} (nonmonotone line search)
\\
If $f\big({\rm Proj}_{\B_{n,r}}(R_k-\tau g_k)\big)
\leq \max_{0\leq j\leq \min\{k,M\}}(f_{k-j})-\gamma \tau \inprod{g_k}{g_k}$, set $\tau_k=\tau$, $R_{k+1}={\rm Proj}_{\B_{n,r}}(R_k-\tau_k g_k),$ and goto Step 6;
Otherwise, set $\tau \leftarrow\sigma \tau$ and repeat Step 4.
\item{[Step 5]} Let 
 $y_k=g_k-\P_{R_k}(g_{k-1}).$ Set 
$\alpha_{k+1}=|\inprod{g_k}{y_k}|/(\tau_k \inprod{g_k}{g_k})$ or
  $\inprod{y_k}{y_k}/(\tau_k \inprod{g_k}{g_k}),$ and goto Step 1. 
 \end{description}

\medskip
In our code, we set $\gamma=10^{-4}, \epsilon=10^{-10}, \sigma=0.5$. For the parameter $\delta$, we set
\begin{equation}\label{del}
\beta=\begin{cases}
1&\|g_k\|_F>1,\\
\|g_k\|_F^{-1}\ & 10^{-5}\leq \|g_k\|_F\leq 1\\
10^5\ &\ \|g_k\|_F<10^{-5}
\end{cases}.
\end{equation}
Given a tolerance
${\rm tol}>0$,
we stop our algorithm when ${\|\g f(R_k)\|_F}/({1+\|R_k\|_F})<{\rm tol}$. 

Note that although we use nonmonotone linesearch in the above algorithm, the condition $f(R)\leq f_k$ is still satisfied in Algorithm 1. The next theorem shows that, our algorithm terminates in finite iterations\footnote{Here finite iteration means that the outer loop of algorithm 1 is finite. We consider the whole Riemannian optimization method in step 1 as part of one outer iteration.}.

\begin{theo}\label{Riealg}
For any initial point $R_0\in \B_{n,r}$ and any $\delta_0\in (0,1),$ Algorithm 1 will terminate in finite iterations.
\end{theo}

\begin{proof}
Assume by contradiction that Algorithm 1 doesn't terminate. Then we have that $\delta_k\rightarrow 0.$ From Lemma~\ref{round}, for $k$ sufficiently large, $\Ron\big(\widehat{R}_k\big)\in \B_{n,r}.$ Thus Case 1 and Case 2 in Step 2 will not be triggered when $k$ is large enough. Since there are finite many singular points inside $E^n e_1^\top.$ One of them will be visited infinitely many times, that is, there exists $a\in E^n$ and $\{i_k\}_{k\in \mathbb{N}}\subset \mathbb{N}$ such that $\Ron\big(\widehat{R}_{i_k}\big)=ae_1^\top$ for any $k\in \mathbb{N}.$ Also, $ae_1^\top$ is non-optimal. From (\ref{RRclose}) in Lemma~\ref{round}, we have that $f\big( \widehat{R}_{i_k}\big)\rightarrow f\(ae_1^\top\).$ Note that from the condition $f(R)\leq f_k$ in Step 1 and Case 3 in Step 2, we have that $f_{k+1}\leq f_k.$ Therefore, for $k$ sufficiently large, $f_k\leq f(ae_1^\top)-\epsilon(ae_1^\top).$ This contradicts that $f\big( \widehat{R}_{i_k}\big)\rightarrow f\(ae_1^\top\).$
\end{proof}
}

%%%%%%%%%%%%%%
\subsection{Graph multi-equipartition}
In this section, we consider the case where $k\geq 3$. In this case, the spectral upper bound in \eqref{SDPLR} already implies that any feasible point $R$ is a smooth point of 
$\B_{n,r}$, then we simply treat $\B_{n,r}$ as a manifold. We use an augmented Lagrangian method on $\B_{n,r}$ to solve the problem SDPLR1 in \eqref{SDPLR}. Note that recently the convergence of ALM on manifold has been studied in \cite{maniALM,WZW,DingALM}.  We reformulate problem (\ref{SDPLR}) as follows:
\begin{equation}\label{SDPLR1}
\min\left\{ \frac{1}{2}\inprod{L}{RR^\top}
+\delta_{\S_+^r}\(Y\):\ R^\top R+Y=\alpha I,\ R\in \B_{n,r}\right\}.
\end{equation}
The Lagrangian function is:
\begin{equation}\label{lgf}
\L\(R,Y,Z\):=\frac{1}{2}\inprod{L}{RR^\top}
-\inprod{Z}{R^\top R+Y-\alpha I}
+\delta_{\S^r_+}\(Y\)+\delta_{\B_{n,r}}\(R\).
\end{equation}

Let $\Pi_{\S^r_-}(\cdot)$ be the projection operator onto $\S^r_-$, the cone of
$r\times r$ symmetric negative semidefinite matrices.
The reduced augmented Lagrangian function is:
\begin{equation}\label{Augf}
\L_\beta\(R,Z\)=\frac{1}{2}\inprod{L}{RR^\top}
+\frac{1}{2\beta}\( \| \Pi_{\S_-^r}\( Z-\beta\(R^\top R-\alpha I\)\)\|^2-\|Z\|^2\).
\end{equation} 
$$ \nabla_R \L_{\beta}\(R,Z\)=LR-2R\,\Pi_{\S_-^r}\(Z-\beta\(R^\top R-\alpha I\)\).$$
The template of the augmented Lagrangian method is as follows:
\begin{description}
\item[Algorithm 2:] Choose $R_0\in \B_{n,r},$ $Z_0\in \S^r,$ $\beta_0>0.$ Set $k=0.$
\item[Step 1.] $R_k\in \arg\min\left\{ \L_{\beta_k}\(R,Z_{k-1}\):\ R\in \B_{n,r}\right\}.$
\item[Step 2.] $Z_k=\Pi_{\S_-^r}\( Z_{k-1}-\beta_k\(R_k^\top R_k-\alpha I\)\).$ 
\item[Step 3.] Choose $\beta_{k+1}$, set $k\leftarrow k+1.$ Goto Step 1.\\
\end{description}

\begin{prop}
Any saddle point of the Lagrangian function (\ref{lgf}) satisfies the first order KKT condition of (\ref{SDPLR}).
\end{prop}
\begin{proof}
Suppose $\(R,Y,Z\)$ is a saddle point of (\ref{lgf}). We have $$R\in \arg\min_{R'}\left\{ \L\(R',Y,Z\)\right\},\  Y\in \arg\min_{Y'}\left\{ \L\(R,Y',Z\)\right\},\ Z\in \arg\max_{Z'}\left\{ \L\(R,Y,Z'\)\right\},$$
which implies that
\begin{align}
&\P_{R}\(LR-2RZ\)=0\label{Sep15_3}\\
&Z\in \N_{\S^r_+}\(Y\)\label{Sep15_4}\\
&R^\top R+Y-\alpha I=0. \label{Sep20_1}
\end{align}
The condition (\ref{Sep15_4}) implies that $Y\in \S_+^r,$ $Z\in \S_-^r$ and $\<Z,Y\>=0.$ From (\ref{Sep20_1}), $RR^\top -\alpha I\preceq 0$ and so $R\in \Omega_\alpha^2.$ Since $\<Z,Y\>=0$, we have that $\<Z,R^\top R-\alpha I\>=0.$ 
Because $Z,\ R^\top R-\alpha I\preceq 0$, the last complementarity 
condition implies that $Z (R^\top R-\alpha I) = 0$ and hence 
$Z$ and $R^\top R-\alpha I$ are simultaneously diagonalizable.
By considering the SVD of $R$ as in \eqref{SVD}, 
we have $Z=V_1 Z_1 V_1^\top$ for some $Z_1\preceq 0.$ So $RZ = U_1 \sqrt{\alpha}Z_1 V_1^\top\in -\N_{\Omega_\alpha^2}\(R\).$ 
The condition (\ref{Sep15_3}) implies $R\in \B_{n,r}$ and 
$$ J\(LR-2RZ-\dd\(\lambda_{LR-2RZ}^R\)R\)=0,$$ 
where $\lambda_{LR-2RZ}^R$ is defined as in (\ref{lamb}).
Thus,
$$ \(L-\dd\(\lambda_{LR-2RZ}^R\)\)R-e\mu^\top =2RZ\in -\N_{\Omega_\alpha^2}\(R\),$$ 
which implies (\ref{KKT1}), 
where 
$\mu = R^\top \(L-\dd\(\lambda_{LR-2RZ}^R\)\)e/n$.
Therefore, $R,\lambda_{LR-2RZ},\mu$ satisfies the first order KKT condition of (\ref{SDPLR}).
\end{proof}

\medskip
We terminate the augmented Lagrangian method based on the relative dual feasibility,
$$ 
\max\Bigg\{ \frac{\|\P_{R}\(LR-2RZ\)\|_F}{1+\|R\|_F+\|Z\|_F},\ \frac{\|Y-\P_{\S_+^r}\(Y+Z\)\|_F}{1+\|Y\|_F+\|Z\|_F}\Bigg\},
$$ 
and 
the relative primal feasibility, $\|RR^\top +Y-\alpha I\|_F/(1+\norm{Y}_F + \norm{R}_F).$ 
We stop the algorithm if  both the primal and dual feasibility are less than some given tolerance ${\rm tol}$. \\

\section{Numerical experiment}
In this section, we conduct numerical experiments to verify the efficiency of our methods. Note that our goal is not to show that the Riemannian optimization methods in section 4 are efficient since they are already {well studied the literature}. Our purpose is to show that by using the geometric structure of the algebraic variety $\B_{n,r}$, we are able to solve the graph equipartition SDP problem (\ref{SDPLR}) much more efficiently than other methods that do not exploit the
geometric structure of $\B_{n,r}$ or only partially exploiting geometric structure like oblique manifold. All the experiments were run in {\sc Matlab} R2020b on a MacBook Pro with 1.4 GHz Quad-Core Intel Core i5 processor and  8GB RAM.

\subsection{A rank adaptive strategy}
In our implementation, {we choose the parameter $r$} adaptively to {further increase computational efficiency.} We first set $r=k-1+\lceil\sqrt{2(n+1)}\rceil,$ which is the theoretical upper bound. Suppose $R\in \R^{n\times r}$ is an iteration point with singular values $\sigma_1\geq \sigma_2\ldots\geq \sigma_r>0.$ We consider the ratio of them $\big\{\frac{\sigma_i}{\sigma_{i+1}}:\ i\in [r-1]\big\}.$ 
If there exists $i\in [r-1]$ such that $\frac{\sigma_{i}}{\sigma_{i+1}}>10$, then we choose $j= \arg\max\big\{ \frac{\sigma_i}{\sigma_{i+1}}:\ i\in [r-1]\big\}$ and drop
the singular values $\sigma_{j+1},\sigma_{j+2},\ldots,\sigma_{r}$ and their corresponding singular vectors to save storage. We perform such rank reduction check 
in every 10 steps. Note that such a rank adaptive strategy is 
adopted from \cite{adaptive}.

\subsection{Experiments on minimum bisection SDPs}
First, we consider the SDPs coming from minimum bisection problems. 
In \cite{maniALM}, Boumal et al. used an ALM on Riemannian manifold\footnote{Source codes from https://github.com/losangle/Optimization-on-manifolds-with-extra-constraints.} called RALM to solve the minimum bisection SDP problems. Their manifold is $\OB_{n,r}$ and they penalized the extra constraint $R^\top e=0.$ We also compare with the well-known BM-method\footnote{Source codes from https://sburer.github.io/projects.html.}(see \cite{BM1,BM2}) called SDPLR, which directly apply an ALM on the 
factorized model  \eqref{SLR} without using the underlying manifold structure. We name our method AVBB, since we use gradient descent with BB step on an algebraic variety. For RALM, the authors stop the algorithm when the distance of two consecutive iterations is less than $10^{-10}.$ For SDPLR, the authors stop their algorithm when the primal residue is less than some tolerance, which we choose to be $10^{-6}$. For our method, we choose $\delta_0=0.02$ in Algorithm 1. We stop the algorithm in step 1 if the relative Riemannian gradient norm is less than ${\rm tol}=10^{-6}.$ Once we get the solution $R$, define 
\begin{equation}\label{dslack}
S_{\lambda}:=L-\A^*\( \lambda\).
\end{equation}
We check the KKT residues of computed solution $X = RR^\top$ 
for the SDP problem (\ref{SDP1}) as follows:
\begin{align}
&{\rm primal\ residue:}\ {\rm Rp} = \frac{\| \A\(RR^\top\)-b\|_2}{1+\|b\|_2},\quad {\rm dual\ residue:}\ {\rm Rd} =  \frac{\max\Big\{0,-\lambda_{\min}\big( JS_{\lambda_{LR}^R}J\big)\Big\}}{1+\|L\|_F},\notag \\
&{\rm\ complementarity:}\ {\rm Rc} = \frac{\big|\inprod{RR^\top}{S_{\lambda_{LR}^R}}\big|}{1+\|L\|_F},
\end{align}
where $\lambda_{LR}^R$ is defined as in (\ref{lamb}).

Because RALM and SDPLR do not use rank adaptive strategy. We first use our method AVBB to solve the SDP problem and check the KKT residues of the output
 $\hat{R}\in \R^{n\times \tilde{r}}.$ If we are able to solve SDP problem successfully, we choose $r=\tilde{r}+5$ for SDPLR and RALM. Note that by doing this, we already give the 
 other solvers some advantage since $r=\tilde{r}+5$ is usually much smaller than the theoretical upper bound. Also, we choose $r$ to be slightly large than $\tilde{r}$ so that they are less likely to reach a spurious local minimum. For the initialization, we  generate a random matrix, $\texttt{randn(n,r)}$, and compute its retraction onto $\B_{n,r}$ to get an initial point $R_0$ for our algorithm.  For RALM and SDPLR, we use their own initialization. When computing the dual residue, we use PROPACKmod in  \cite{TohYun} that is 
 modified from PROCPACK in \cite{propack}
 to compute the smallest eigenvalue, since it can make use of sparsity and low-rank property.

The graph data set we use in experiments comes from some small dense graphs from \cite{thetalib} and large social networks from \cite{Musae}.

\begin{center}
\begin{footnotesize}
\begin{longtable}{|c|c|cccc|l|}
\caption{Comparison of AVBB, RALM and SDPLR for minimum bisection SDP.}
\label{BC1}
\\
\hline
problem & algorithm & Rp & Rd & Rc& obj & time \\ \hline
\endhead
brock200-1 & AVBB & 1.97e-16 & 1.28e-10 & 8.62e-16 & 3.9857926e+03 & 2.95e-01 \\
 n=200& RALM& 1.28e-06 & 5.46e-10 &1.51e-06 &3.9857915e+03 &1.64e+01 \\
 m=5067 & SDPLR&1.32e-07&4.75e-10&4.93e-07&3.9857923e+03&3.58e+00\\
 \hline
 brock200-4 & AVBB & 1.72e-16  &  1.31e-10 & 1.02e-15  & 5.6562152e+03 &  1.56e-01 \\
 n=200& RALM& 1.21e-07 & 2.17e-09 &8.35e-07 &5.6562144e+03&1.90e+01 \\
 m=6812 & SDPLR& 1.32e-07 & 6.44e-11 &8.56e-07&5.6562143e+03 &7.52e+00\\
 \hline

brock400-1 & AVBB &  1.95e-16  &  4.84e-10 & 1.72e-15  & 1.6933377e+04 &  3.43e-01 \\
 n=400& RALM& 1.50e-07 & 1.10e-08 & 3.72e-07 &1.6933376e+04&2.55e+01 \\
 m=20078 & SDPLR& 9.51e-08 &3.40e-10 & 2.97e-07&1.6933377e+04  &3.04e+01\\
 \hline

c-fat200-1 & AVBB &  1.57e-16  &  1.09e-10 & 1.18e-15  & 1.7825896e+04 & 1.07e-01 \\
 n=200& RALM& 1.09e-06 & 7.39e-12 & 3.24e-07 & 1.7825895e+04 & 7.62e+00 \\
 m=18367 & SDPLR&1.31e-07 & 9.08e-11 & 1.05e-07 & 1.7825896e+04 & 3.06e+00\\
 \hline

hamming-6-4 & AVBB &  1.72e-16  & 1.56e-09 & 5.78e-16  & 1.0240000e+03 &  6.98e-02 \\
 n=64& RALM& 1.81e-16 & 2.97e-10 & 7.15e-16 & 1.0240000e+03 & 3.41e-01 \\
 m=1313 & SDPLR&1.08e-07 & 1.47e-06 & 2.40e-07 & 1.0240000e+03 & 3.20e-01\\
 \hline

hamming-7-5-6 & AVBB &  2.16e-16  &  4.23e-09 & 1.12e-15  & 1.5360000e+03 &  5.51e-02\\
 n=128& RALM& 2.14e-16 & 1.26e-10 & 9.78e-16 & 1.5360000e+03 & 1.02e+00\\
 m=1793 & SDPLR&1.50e-07 & 2.24e-07 & 1.09e-07 & 1.5360000e+03 & 3.44e-01\\
 \hline

hamming-8-3-4 & AVBB & 1.83e-16  &  2.16e-10 & 1.36e-15  & 1.4336000e+04 &  5.06e-02\\
 n=256& RALM& 2.41e-13 & 1.48e-10 & 2.92e-15 & 1.4336000e+04 & 2.08e+00\\
 m=16129 & SDPLR&3.18e-08 & 9.79e-08 & 1.80e-07 & 1.4336000e+04 & 4.20e-01\\
 \hline

 hamming-8-4& AVBB &   2.01e-16  &  9.02e-11 & 8.05e-16  & 7.4240000e+03 &  6.57e-02\\
 n=256& RALM& 3.35e-16 & 1.67e-11 & 1.15e-15 & 7.4240000e+03 & 1.40e+00\\
 m=11777 & SDPLR&5.76e-08 & 1.95e-07 & 2.22e-08 & 7.4240000e+03 & 4.50e-01\\
 \hline

 hamming-9-5-6& AVBB &   2.22e-16  &  1.81e-10 & 2.42e-15  & 5.0176000e+04 &  1.37e-01\\
 n=512& RALM& 2.69e-14 & 2.87e-11 & 3.12e-15 & 5.0176000e+04 & 5.41e+00\\
 m=53761 & SDPLR&2.63e-08 & 8.32e-08 & 4.17e-08 & 5.0176000e+04 & 7.59e-01\\
 \hline

 hamming-9-8& AVBB &  0.00e+00  &  -0.00e+00 & 0.00e+00  & 0.0000000e+00 &  3.56e-01\\
 n=512& RALM&2.22e-16 & 4.01e-19 & 1.40e-15 & 1.1220542e-14 & 3.24e+00\\
 m=2305 & SDPLR&3.32e-08 & 0.00e+00 & 1.11e-10 & 1.1938166e-08 & 6.37e-01\\
 \hline

hamming-10-2& AVBB &0.00e+00  &  -0.00e+00 & 0.00e+00  & 0.0000000e+00 &  6.32e-01\\
 n=1024& RALM&2.22e-16 & 0.00e+00 & 3.04e-15 & 1.8270875e-14 & 5.61e+00\\
 m=23041 & SDPLR&1.01e-08 & 5.07e-14 & 4.43e-11 & 3.2290904e-08 & 1.26e+00\\
 \hline

hamming-11-2& AVBB & 0.00e+00  &  -0.00e+00 & 0.00e+00  & 0.0000000e+00 &  6.95e-01\\
 n=2048& RALM&2.27e-16 & 7.55e-18 & 1.56e-14 & 2.7414197e-12 & 2.28e+01\\
 m=56321 & SDPLR&2.96e-08 & 1.56e-18 & 3.83e-11 & 4.8081223e-08 & 5.05e+00\\
 \hline

hamming6-2& AVBB &1.28e-16  &  3.03e-09 & 5.47e-17  & 6.4000000e+01 &  4.40e-02\\
 n=64& RALM&4.47e-16 & 1.31e-08 & 2.33e-16 & 6.4000000e+01 & 1.67e+00\\
 m=193 & SDPLR&1.85e-07 & 8.43e-08 & 2.39e-07 & 6.4000019e+01 & 4.33e-01\\
 \hline

hamming8-2-G& AVBB &1.96e-16  &  4.60e-09 & 3.91e-16  & 2.5600000e+02 &  6.54e-02\\
 n=256& RALM&1.84e-16 & 2.00e-10 & 4.27e-16 & 2.5600000e+02 & 2.28e+00\\
 m=1024 & SDPLR&9.34e-08 & 1.02e-06 & 1.07e-08 & 2.5600000e+02 & 5.53e-01\\
 \hline

MANN-a27& AVBB & 2.36e-16  &  3.63e-09 & 2.49e-16  & 1.2219459e+02 &  2.95e-01\\
 n=378& RALM&2.22e-16 & 1.73e-10 & 2.24e-16 & 1.2219459e+02 & 2.36e+01\\
 m=703 & SDPLR&8.08e-08 & 7.13e-07 & 1.13e-07 & 1.2219460e+02 & 7.51e-01\\
 \hline

johnson8-4-4 & AVBB &  1.86e-16  &  1.18e-09 & 3.92e-16  & 2.8000000e+02 &  2.86e-02 \\
 n=70& RALM& 1.73e-16 & 8.38e-10 & 4.85e-16 & 2.8000000e+02 & 6.10e-01 \\
 m=561 & SDPLR&1.12e-07 & 1.55e-06 & 5.82e-08 & 2.8000000e+02 & 4.20e-01\\
 \hline

johnson16-2-4 & AVBB & 1.98e-16  &  1.34e-11 & 4.06e-16  & 9.6000000e+02 &  4.86e-02 \\
 n=120& RALM& 1.85e-16 & 1.01e-10 & 8.23e-16 & 9.6000000e+02 & 1.09e+00 \\
 m=1681 & SDPLR&6.41e-08 & 7.84e-07 & 1.62e-07 & 9.5999998e+02 & 3.38e-01\\
 \hline

keller4 & AVBB &  1.88e-16  &  1.42e-09 & 4.34e-16  & 3.2898566e+03 &  1.12e-01 \\
 n=171& RALM& 7.57e-06 & 1.56e-11 & 1.08e-05 & 3.2898481e+03 & 4.66e+00\\
 m=5101 & SDPLR&1.42e-07 & -0.00e+00 & 1.47e-06 & 3.2898555e+03 & 7.43e+00\\
 \hline

keller5-G & AVBB & 1.97e-16  &  4.69e-10 & 2.17e-15  & 4.7829110e+04 &  5.65e-01 \\
 n=776& RALM& 1.76e-07 & 4.30e-10 & 3.51e-07 & 4.7829108e+04 & 3.16e+01\\
 m=74710 & SDPLR&6.92e-08 & 3.79e-11 & 2.20e-07 & 4.7829109e+04 & 4.45e+01\\
 \hline

 p-hat300-1& AVBB & 1.75e-16  &  4.13e-11 & 1.49e-15  & 3.2077515e+04 &  1.71e-01 \\
 n=300& RALM& 7.60e-07 & 1.00e-09 & 1.10e-06 & 3.2077511e+04 & 1.66e+01\\
 m=33918 & SDPLR&1.09e-07 & 6.84e-11 & 4.14e-07 & 3.2077514e+04 & 2.56e+01\\
 \hline

san200-07-1& AVBB & 1.84e-16  &  1.04e-09 & 1.19e-15  & 4.8782920e+03 &  1.84e-01 \\
 n=200& RALM& 1.99e-07 & 3.22e-09 & 8.08e-07 & 4.8782913e+03 & 1.65e+01\\
 m=5971 & SDPLR&1.31e-07 & 2.78e-10 & 6.54e-07 & 4.8782914e+03 & 6.26e+00\\
 \hline

musae-PTBR& AVBB &1.67e-16  &  1.21e-10 & 1.09e-14  & 1.0504381e+04 &  4.64e+00\\
 n=1912& RALM& 5.45e-08 & 4.26e-06 & 5.48e-06 & 1.0504640e+04 & 6.85e+01\\
 m=31299 & SDPLR&5.98e-07 & 1.09e-10 & 1.84e-05 & 1.0504329e+04 & 4.70e+02\\
 \hline

musae-chameleon& AVBB &1.42e-16  &  1.28e-10 & 3.18e-15  & 4.7543522e+02 &  2.40e+01\\
 n=2277& RALM& 3.19e-09 & 2.04e-05 & 1.03e-07 & 4.7833683e+02 & 8.82e+01\\
 m=36101 & SDPLR&4.79e-08 & 4.77e-19 & 4.92e-07 & 4.7543371e+02 & 4.76e+02\\
 \hline

musae-RU& AVBB &1.59e-16  &  5.82e-11 & 9.31e-16  & 1.1774237e+04 &  1.49e+01\\
 n=4385& RALM& - & - & - & - & -\\
 m=37304 & SDPLR&1.97e-05 & 9.37e-09 & 5.47e-05 & 1.1774072e+04 & 4.77e+02\\
 \hline

musae-ES& AVBB &1.67e-16  &  1.55e-11 & 6.02e-16  & 1.5038307e+04 &  6.45e+00\\
 n=4648& RALM& - & - & - & - & -\\
 m=59482 & SDPLR&7.30e-05 & 7.19e-07 & 9.63e-05 & 1.5038055e+04 & 5.02e+02\\
 \hline

musae-squirrel& AVBB & 1.66e-16  &  3.65e-10 & 8.89e-15  & 1.0907331e+04 &  3.77e+01\\
 n=5201& RALM& - & - & - & - & -\\
 m=217073 & SDPLR&3.32e-04 & 1.78e-07 & 9.43e-04 & 1.0899471e+04 & 5.06e+02\\
 \hline

musae-FR& AVBB &1.55e-16  &  2.56e-05 & 1.62e-15  & 3.6960466e+04 &  2.04e+01\\
 n=6549& RALM& - & - & - & - & -\\
 m=112666 & SDPLR&4.01e-02 & 4.41e-05 & 2.24e-01 & 3.4091560e+04 & 5.01e+02\\
 \hline

musae-ENGB& AVBB &1.71e-16  &  7.65e-10 & 1.06e-15  & 8.9985089e+03 &  1.04e+01\\
 n=7126& RALM& - & - & - & - & -\\
 m=35324 & SDPLR&9.42e-03 & 4.88e-06 & 8.45e-02 & 8.8164546e+03 & 4.96e+02\\
 \hline

musae-DE& AVBB &1.53e-16  &  4.69e-11 & 4.17e-16  & 4.9503018e+04 &  2.88e+01\\
 n=9498& RALM& - & - & - & - & -\\
 m=153138 & SDPLR&- & - & -& - &-\\
 \hline

musae-crocodile & AVBB & 1.83e-16  &  9.02e-12 & 1.70e-15  & 1.3265035e+04 &  3.32e+01 \\
 n=11631& RALM& - & - & - & - & -\\
 m=180020 & SDPLR&- & - & - & -& -\\
 \hline

musae-facebook & AVBB & 1.96e-16  &  1.20e-09 & 3.17e-15  & 6.0495334e+03 &  1.12e+02 \\
 n=22470& RALM& - & - & - & - & -\\
 m=171002 & SDPLR&- & - & - & -& -\\
 \hline
 
\end{longtable}
\end{footnotesize}
\end{center}
%%%%%%%%%%%%%%%%%%%%%%%%%%%%%%%%%%
%%%%%%%%%%%%%%%%%%%%%%%%%%%%%%%%%%

From Table 1, we can see that if the graph size is moderate, all the three methods can find optimal solution for the SDP \eqref{SDP1} successfully
since its corresponding KKT residues (Rp,Rd,Rc) are all smaller than the required tolerance.
However, when the graph size is too large, such as
{\tt musae-DE}, {\tt musae-crocodile} and {\tt musae-facebook}, only AVBB can solve these problems. 
In particular, we are able to solve the largest instance {\tt musae-facebook} with 
$n=22470$ and $m=171002$ in about two minutes.

Both RALM and SDPLR will terminate if there is little progress for many iterations.
In the tables, ``-" means that they terminate prematurely and return a solution that is 
far from the optimal solution. Among all the problems, AVBB is clearly
 faster and more accurate than RALM and SDPLR. Note that for 
 {\tt hamming-9-8}, {\tt hamming-10-2} and {\tt hamming-11-2}, the optimal solution is a singular point and AVBB finds them successfully.

%%%%%%%%%%%%%%%%%%%%%%%%%%%%%%%%%%
\subsection{Experiments on graph equipartition SDPs}

In this section, we test on graph equipartition SDP problems with more than 2 partitions. Since RALM and SDPLR cannot handle the extra SDP upper bound constraint
$X\preceq \alpha I$, we do not test them in the experiments. 
Instead, we compare our method with the interior point method (which we denote as IPM) in section 6 of \cite{TTT} that can handle the SDP upper bound constraint directly. 
{Note that applying the IPM in \cite{TTT} to handle the semidefinite upper bound constraint directly is much more efficient than applying an IPM to the reformulated problem of converting the bound constraint to an affine constraint with an additional slack SDP variable.}

We call our method AVALM where ``AV" stands for ``algebraic variety''. We use a Riemannian gradient method with BB step and nonmonotone linesearch to solve the ALM subproblem to the accuracy of $10^{-6}$ in the relative norm of 
 the Riemannian gradient. We also set the maximum number of iterations for solving the subproblem to be 200. We set the initial penalty parameter $\beta=0.1$ and set 
$$ 
\beta_{k+1}=\begin{cases} \max\{\beta_k/1.2,0.1\}&{\rm pfeas < dfeas/1000}\\ \min\{1.2\beta_k,10\} &{\rm pfeas \geq  \max\{dfeas/1000,10\cdot tol\}}.\end{cases}
$$ 

\noindent We stop our ALM algorithm if both the primal feasibility (pfeas) and dual feasibility (dfeas)
 are less than $10^{-6}.$
We also stop our algorithm if there is no progress for many iterations. After we get a solution $(R,Z)$, let ${\bf S}=S_{\lambda_{LR-2RZ}}$ as in (\ref{dslack}), then we check the following relative KKT residues for the original convex SDP problem (\ref{SDP1}):
\begin{align}
&{\rm primal\ residue:}\ {\rm Rp} = \max\left\{ \frac{\| \A\(RR^\top\)-b\|_2}{1+\|b\|_2},\|R\|_2-\sqrt{\alpha}\right\},\notag \\ 
&{\rm dual\ residue:}\ {\rm Rd}=\frac{\max\left\{\lambda_{\max}\(U_1^\top {\bf S} U_1\),\ \max\left\{-\lambda_{\min}\(J\({\bf S}-U_1U_1^\top {\bf S} U_1U_1^\top\)J\),0\right\}\right\}}{1+\|L\|_F}, \notag \\
&{\rm complementarity:}\  {\rm Rc} =   \frac{\Big|  \< {\bf S}-U_1U_1^\top {\bf S} U_1U_1^\top,RR^\top\>\Big|}{1+\|L\|_F},
\end{align}
where $U_1$ is defined in (\ref{SVD}). If $\max\{{\rm Rp},{\rm Rd},{\rm Rc}\}$ is small, then we have
obtained an accurate approximate optimal solution to the SDP problem \eqref{SDP1}.

Since the interior point method IPM cannot handle graph equipartition SDP problems
arising from 
large graphs, we only compare AVALM with IPM on small datasets in \cite{thetalib},
where we set the number of partitions $k=5.$ From 
Table \ref{TEQP1}, we see that both AVALM and IPM can solve all of these problems accurately. However, AVALM is much more efficient than IPM, and for some instances AVALM can be 20-100 times faster.

%%%%%%%%%%%%%%%%%%%%%%%%%%%%%%%%
\begin{center}
\begin{footnotesize}
\begin{longtable}{|c|c|cccc|l|}
\caption{Comparison of AVALM with IPM for graph equipartition SDP with $k=5$.}
\label{TEQP1}
\\
\hline
problem & algorithm & Rp & Rd & Rc& obj & time \\ \hline
\endhead
brock200-1 & AVALM & 6.22e-06 & 9.53e-09 & 2.12e-08 & 3.9860693e+03 & 4.54e-01 \\
 n=200,m=5067& IPM& 1.21e-07 & 2.37e-07 & 2.14e-06 & 3.9860704e+03 & 3.90e+00
 \\
 \hline
 brock200-4 & AVALM & 4.15e-06 & 4.47e-08 & 9.17e-11 & 5.6564739e+03 & 4.40e-01 \\
 n=200,m=6812& IPM& 1.34e-07 & 5.56e-09 & 1.37e-06 & 5.6564750e+03 & 3.88e+00 \\
 \hline

brock400-1 & AVALM &  2.15e-16 & 1.12e-09 & 2.15e-15 & 1.6933377e+04 & 3.55e-01 \\
 n=400,m=20078& IPM& 1.93e-07 & 8.08e-10 & 6.99e-06 & 1.6933383e+04 & 7.52e+00 \\
 \hline

c-fat200-1 & AVALM &  1.47e-06 & 8.34e-16 & 5.45e-18 & 1.7842168e+04 & 5.48e-01 \\
 n=200,m=18367& IPM& 7.51e-08 & 1.22e-10 & 1.92e-06 & 1.7842169e+04 & 1.96e+00 \\
 \hline

hamming-6-4 & AVALM &  1.63e-16 & 1.01e-09 & 3.73e-16 & 1.0240000e+03 & 1.94e-02 \\
 n=64,m=1313& IPM& 2.41e-09 & 1.63e-08 & 4.39e-08 & 1.0240002e+03 & 2.77e-01 \\
  \hline

hamming-7-5-6 & AVALM &  2.23e-16 & 2.99e-09 & 9.81e-16 & 1.5360000e+03 & 3.11e-02\\
 n=128,m=1793& IPM& 6.72e-09 & 1.12e-10 & 3.73e-07 & 1.5360001e+03 & 6.52e-01\\
  \hline

hamming-8-3-4 & AVALM &  2.01e-16 & 6.37e-10 & 1.55e-15 & 1.4336000e+04 & 5.10e-02\\
 n=256,m=16129& IPM& 2.99e-08 & 4.08e-11 & 8.95e-08 & 1.4336000e+04 & 1.31e+00\\
 \hline

 hamming-8-4& AVALM&  2.05e-16 & 6.60e-11 & 8.56e-16 & 7.4240000e+03 & 5.39e-02\\
 n=256& IPM& 1.74e-07 & 1.66e-10 & 1.61e-07 & 7.4240001e+03 & 1.32e+00\\
  \hline

 hamming-9-5-6& AVALM &2.55e-16 & 1.70e-10 & 2.23e-15 & 5.0176000e+04 & 1.12e-01\\
 n=512,m=53761& IPM& 8.70e-09 & -0.00e+00 & 2.15e-06 & 5.0176005e+04 & 3.74e+00\\
  \hline

 hamming-9-8& AVALM & 2.38e-08 & 5.54e-09 & 1.41e-09 & 7.6800000e+02 & 3.08e-01\\
 n=512,m=2305& IPM&3.50e-07 & 1.00e-13 & 5.65e-08 & 7.6800002e+02 & 5.78e+00\\
 \hline

hamming-10-2& AVALM &7.25e-09 & 3.93e-10 & 8.14e-11 & 6.9120000e+03 & 4.19e-01\\
 n=1024,m=23041& IPM&5.19e-08 & 0.00e+00 & 3.72e-07 & 6.9120002e+03 & 2.69e+01\\
  \hline

hamming-11-2& AVALM &2.12e-16 & 3.04e-09 & 1.37e-08 & 1.5360000e+04 & 8.34e-01\\
 n=2048,m=56321& IPM&4.05e-07 & 1.00e-12 & 2.87e-07 & 1.5360001e+04 & 2.10e+02\\
 \hline

hamming6-2& AVALM &4.75e-06 & 4.50e-08 & 6.57e-08 & 6.4711622e+01 & 2.12e-01\\
 n=64,m=193& IPM&5.12e-08 & 9.33e-09 & 9.55e-08 & 6.4711622e+01 & 6.46e-01\\
  \hline

hamming8-2-G& AVALM &1.89e-16 & 1.16e-08 & 4.04e-16 & 2.5600000e+02 & 5.53e-02\\
 n=256,m=1024& IPM&1.63e-08 & 2.02e-10 & 2.76e-07 & 2.5600002e+02 & 1.27e+00\\
  \hline

MANN-a27& AVALM &2.30e-16 & 2.58e-08 & 3.89e-16 & 1.2219459e+02 & 3.97e-01\\
 n=378,m=703& IPM&4.14e-08 & -0.00e+00 & 9.37e-07 & 1.2219464e+02 & 3.05e+00\\
  \hline

johnson8-4-4 & AVALM & 1.85e-16 & 1.16e-09 & 9.12e-09 & 2.8000000e+02 & 2.13e-02 \\
 n=70,m=561& IPM& 4.01e-09 & 9.61e-10 & 1.12e-08 & 2.8000001e+02 & 2.05e-01\\
 \hline
 johnson16-2-4 & AVALM & 2.21e-16 & 5.68e-10 & 9.10e-16 & 9.6000000e+02 & 2.68e-02 \\
 n=120,m=1681& IPM& 1.73e-08 & 2.92e-18 & 1.14e-06 & 9.6000017e+02 & 4.28e-01 \\
 \hline

keller4 & AVALM &  5.62e-06 & 2.77e-13 & 1.39e-08 & 3.2960645e+03 & 1.82e-01 \\
 n=171,m=5101& IPM& 1.03e-07 & 1.39e-09 & 1.49e-06 & 3.2960652e+03 & 1.91e+00\\
 \hline

keller5-G & AVALM &  1.94e-16 & 5.93e-10 & 2.26e-15 & 4.7829110e+04 & 9.28e-01 \\
 n=776,m=74710& IPM& 3.37e-03 & 1.04e-13 & 5.75e-08 & 4.7855470e+04 & 3.56e+01 \\
 \hline

p-hat300-1 & AVALM &  4.07e-06 & 4.19e-10 & 5.46e-10 & 3.2077517e+04 & 2.83e-01\\
 n=300,m=33918& IPM&3.43e-07 & 2.84e-09 & 3.57e-06 & 3.2077523e+04 & 5.83e+00\\
  \hline

san200-07-1 & AVALM & 2.06e-06 & 8.32e-07 & 3.42e-09 & 4.8786431e+03 & 4.26e-01\\
 n=200,m=5971& IPM& 8.72e-09 & 8.93e-09 & 3.42e-06 & 4.8786445e+03 & 4.61e+00\\
  \hline
 
\end{longtable}
\end{footnotesize}
\end{center}

%%%%%%%%%%%%%%%%%%%%%%%%%%%%%%

We also test AVALM alone on
graph equipartition SDPs arising large social network graphs 
in \cite{Musae}, where we choose $k=3,5,10$ respectively.
Table \ref{TEQP3}
clearly shows that our algorithm can solve all of these large scale SDP problems  
extremely efficiently and accurately 
with certified global optimality for the SDP problem \eqref{SDP1} 
based on ${\rm Rp}$, ${\rm Rd}$ 
and ${\rm Rc}$.
For the largest instance {\tt musae-facebook} with $n=22470$ and $m=171002$, we are able
to solve the SDP problem 
in under 5 minutes when $k=3$ or $5$, and in 
about 11 minutes when $k=10$.

\begin{center}
\begin{footnotesize}
\begin{longtable}{|c|c|cccc|c|} 
\caption{Test results of AVALM on large graph equipartition SDP.}
\label{TEQP3}
\\
\hline
problem & k & Rp & Rd & Rc& obj & time \\ \hline
\endhead
musae-PTBR& 3 &1.26e-05 & 1.08e-09 & 3.88e-10 & 1.1097575e+04 & 2.84e+00\\
 n=1912& 5& 3.53e-08 & 2.30e-10 & 3.67e-08 & 1.3089330e+04 & 6.41e+00\\
 m=31299 & 10&2.09e-07 & 1.23e-08 & 1.20e-06 & 1.5773687e+04 & 1.25e+01\\
 \hline

musae-chameleon& 3 &9.60e-08 & 9.21e-10 & 2.61e-08 & 1.2425855e+03 & 3.94e+00\\
 n=2277& 5& 2.07e-06 & 1.21e-09 & 5.51e-09 & 2.2463854e+03 & 3.77e+00\\
 m=36101 & 10&3.17e-07 & 2.08e-07 & 7.58e-07 & 4.2365961e+03 & 2.85e+01\\
 \hline

musae-RU& 3 &1.69e-16 & 5.20e-11 & 2.36e-09 & 1.2051705e+04 & 7.31e+00\\
 n=4385& 5&1.18e-07 & 2.63e-08 & 8.56e-09 & 1.3456477e+04 & 1.07e+01\\
 m=37304 & 10&1.75e-07 & 9.25e-10 & 1.41e-09 & 1.5615746e+04 & 1.43e+01\\
 \hline

musae-ES& 3 &6.99e-06 & 4.86e-10 & 2.68e-09 & 1.5267811e+04 & 4.89e+00\\
 n=4648& 5& 8.00e-08 & 1.04e-09 & 3.05e-08 & 1.8335808e+04 & 5.57e+00\\
 m=59482 & 10&1.25e-07 & 1.31e-08 & 6.58e-08 & 2.3154327e+04 & 3.18e+01\\
 \hline

musae-squirrel& 3 &1.61e-07 & 4.92e-08 & 1.63e-09 & 1.6163193e+04 & 2.71e+01\\
 n=5201& 5& 3.27e-06 & 2.50e-09 & 5.44e-07 & 2.8359632e+04 & 5.00e+01\\
 m=217073 & 10&4.14e-06 & 8.00e-08 & 2.39e-07 & 5.8591875e+04 & 2.54e+02\\
 \hline

musae-FR& 3 &4.52e-08 & 9.30e-09 & 8.35e-08 & 3.7199157e+04 & 1.38e+01\\
 n=6549& 5& 8.64e-07 & 9.50e-11 & 4.35e-10 & 4.1936390e+04 & 9.01e+00\\
 m=112666 & 10&7.28e-07 & 7.67e-10 & 2.52e-09 & 5.0585680e+04 & 3.55e+01\\
 \hline

musae-ENGB& 3 &1.90e-16 & 2.35e-09 & 7.57e-16 & 8.9985089e+03 & 8.38e+00\\
 n=7126& 5& 3.56e-06 & 4.56e-08 & 3.70e-07 & 9.3864146e+03 & 1.67e+01\\
 m=35324 & 10&1.04e-04 & 3.05e-08 & 1.70e-06 & 1.0792772e+04 & 2.85e+01\\
 \hline

musae-DE& 3 &1.95e-16 & 5.19e-10 & 1.99e-06 & 5.1971122e+04 & 2.26e+01\\
 n=9498& 5& 6.75e-07 & 1.06e-10 & 1.88e-09 & 6.0468739e+04 & 1.73e+01\\
 m=153138 & 10&1.58e-07 & 5.76e-10 & 3.24e-07 & 7.1324948e+04 & 1.06e+02\\
 \hline

musae-crocodile& 3 &5.28e-07 & 4.78e-11 & 5.25e-10 & 1.3265057e+04 & 3.11e+01\\
 n=11631& 5& 2.73e-07 & 3.69e-10 & 1.04e-07 & 1.6310440e+04 & 5.30e+01\\
 m=180020 & 10&3.56e-05 & 9.59e-10 & 2.30e-06 & 2.8584887e+04 & 2.32e+02\\
 \hline

musae-facebook& 3 &1.82e-16 & 8.27e-09 & 1.59e-15 & 6.0495349e+03 & 1.19e+02\\
 n=22470& 5& 2.50e-06 & 5.50e-09 & 1.89e-08 & 6.7587115e+03 & 2.42e+02\\
 m=171002 & 10&1.98e-16 & 1.62e-09 & 2.04e-09 & 9.6414771e+03 & 6.53e+02\\
 \hline

\end{longtable}
\end{footnotesize}
\end{center}

\section{Conclusion}
In this paper, we study the SDP relaxation of a graph equipartition problem. We study the optimality condition of the low-rank factorization model  of the SDP problem, which contains an 
additional SDP upper bound as compared to the standard linear SDP problem. We prove the equivalence between the SDP problem and its factorized version. In addition, 
we prove that under the constraint nondegeneracy condition, any rank-deficient {second order stationary point} of the factorized problem is a global optimal solution. To solve the SDP problem, we study the properties of a special algebraic variety $\B_{n,r}$, which comes from adding an extra coupling affine constraint to the oblique manifold $\OB_{n,r}.$ We find a closed form solution of the projection mapping onto $\B_{n,r}$, and more importantly, we prove that the retraction is equivalent to a geometric median problem under some condition. We also study the local geometric properties of singular points on $\B_{n,r}.$ With these properties, we are able to use a gradient descent method with BB step and nonmonotone line search and an augmented Lagrangian method on the algebraic variety $\B_{n,r}$ to solve the factorized SDP problem efficiently. The results from our numerical experiments verify the high efficiency of our method. {Our work shows that by making use of the geometric property of the feasible region, one can design algorithms that are significantly faster than other algorithms that do not rely on the geometric property like SDPLR or only partially rely on its geometric property like RALM.}

\section*{Acknowledgments}

{We thank the editors and reviewers 
for their valuable suggestions and comments, which have helped
to improve the quality of this paper.}

%%%%%%%%%%%%%%%%%%%%%%%%%%%%%%%%%%%

\newpage
\appendix
%%%%%%%%%%%%%%%%%%%%%%%%%%%%%%%Auxiliary%%%%%%%%%%%%%%%%%%%%%%%%%%%%%

{
\section{Useful auxiliary results}
\subsection{A corollary of the implicit function theorem}
\begin{lem}\label{implicit}
Let $F:\R^{n}\times \R^{m}\rightarrow \R^{k}$ be a continuously differentiable mapping in a neighbourhood of $\( x_0,y_0\)\in \R^{n}\times \R^{m}.$ Suppose $F(x_0,y_0)=0$ and $\D_y F(x_0,y_0):\R^m\rightarrow \R^k$ is a surjective linear mapping. Then there exists $\delta>0$ and a continuously differentiable mapping $y:B_\delta(x_0)\rightarrow \R^k$ such that $y(x_0)=y_0$ and $F(x,y(x))=0.$ Here, $B_\delta(x_0):=\left\{ x\in \R^n:\ \dist(x,x_0)<\delta\right\}.$
\end{lem}
\begin{proof}
Since $\D_y F(x_0,y_0)$ is surjective, we have that $m\geq k.$ If $m=k,$ then Lemma~\ref{implicit} directly follows from the implicit function theorem and the mapping $y$ is unique. If $m>k$, we may find a linear mapping $\A:\R^m\rightarrow \R^{m-k}$ such that the linear operator $\triangle y\rightarrow \(\D_y F(x_0,y_0)[\triangle y],\A(\triangle y)\)$ is a bijective mapping from $\R^m$ to $\R^m.$ Thus, we can define a new mapping $\widehat{F}:\R^{n}\times \R^{m}\rightarrow \R^m$ such that $\widehat{F}(x,y):=\(F(x,y),\A(y)-\A(y_0)\)$ and apply the implicit function theorem to $\widehat{F}.$  
\end{proof}

\subsection{Properties of the rounding procedure}
\begin{lem}\label{round}
Suppose $n\geq 2$ and $\delta<\frac{1}{2}.$ Then for any $R\in \B_{n,r}^{\delta+},$
\begin{equation}\label{RRclose}
\left\| \Ron\(R\)\Ron\(R\)^\top-RR^\top\right\|\leq 2\sqrt{\delta}n.
\end{equation}
Moreover, if $\delta<\frac{1}{n^2},$ then $\Ron(R)\in \B_{n,r}.$
\end{lem}

\begin{proof}
Let $s\in \R$ be the unique largest singular value of $R.$ Its corresponding singular vectors are $u\in \R^n$ and  $v\in \R^r.$ Then we have that 
\begin{equation}\label{May_7_1}
\| usv^\top-R\|^2=\|R\|^2-s^2=n-\|R\|_2^2\leq \delta n.
\end{equation}
Thus,
\begin{align}\label{May_7_2}
&\|\sgn(u)-us\|=\|\sgn(u)v^\top-usv^\top\|\leq \dist\(usv^\top,\OB_{n,r}\)\notag \\
&\leq \dist\(usv^\top,\B_{n,r}\)\leq \| usv^\top-R\|\leq \sqrt{\delta}\sqrt{n},
\end{align}
where the first inequality comes from the retraction of oblique manifold, the second inequality comes from $\B_{n,r}\subset \OB_{n,r},$ the third inequality comes from $R\in \B_{n,r}.$  We have that,
\begin{align}\label{May_7_3}
&\left\| \Ron\(R\)\Ron\(R\)^\top-RR^\top\right\|=\left\| \sgn(u)\sgn(u)^\top-us^2u^\top\right\|\notag \\
&\leq \left\| \sgn(u)\sgn(u)^\top-\sgn(u)su^\top\right\|+\left\|\sgn(u)su^\top-us^2u^\top\right\|\notag \\
&=\left\| \sgn(u)\right\| \left\|\( \sgn(u)^\top-su^\top\)\right\|+\left\| \(\sgn(u)-us\)\right\| \left\|su^\top\right\|\notag \\
&\leq 2\sqrt{n}.\sqrt{\delta}\sqrt{n}=2\sqrt{\delta}n.
\end{align}

Now suppose that $\delta<\frac{1}{n^2}.$ From (\ref{May_7_1}), $\| usv^\top-R\|^2<1.$ Since every row of $R$ has unit length, we have that every entry of $u$ is nonzero. Thus, $\sgn(u)\in \{-1,1\}^n$ and $\Ron(R)=\sgn(u)e_1^\top\in \OB_{n,r}.$ From (\ref{May_7_2}), we have that
\begin{equation}\label{May_7_4}
\left|e^\top \sgn(u)\right|=\left|e^\top\( \sgn(u)-us\)\right|\leq \sqrt{n}\|\sgn(u)-us\|\leq n\sqrt{\delta}<1,
\end{equation}
where the first equality comes from that because $R^\top e=0$ and $u$ is an singular vector of $R,$ $e^\top u=0.$ Because $\left|e^\top \sgn(u)\right|$ is an integer, we have that $\left|e^\top \sgn(u)\right|=0$ and so $\sgn(u)\in E^n.$ Therefore, $\Ron(R)=\sgn(u)e_1^\top\in \B_{n,r}.$
\end{proof}

%%%%%%%%%%%%%%%%%%%%%%%%%%%%%%%%%%%%%
\section{Escaping from a non-optimal singular point of $\B_{n,r}$}

In this section, we consider problem (\ref{SDP1}) with $\A=\dd(.)$ and its low rank version (\ref{GEPLR}) without the spectral upper bound,
i.e., the problems \eqref{BCSDP} and \eqref{BCsmo} in 
subsection 4.1.

%The reason we remove the upper bound is that it only appears in graph equipartition with more than 2 parts. In the latter case, the spectral upper bound already implies that every feasible point $R$ is smooth on $\B_{n,r}$ (see Remark~\ref{remretrac}). 
%We rewrite the problems as follows:
%
%\begin{equation}\label{BCSDP}
%\min\left\{ \<C,X\>:\ \dd\(X\)=e,\ Xe=0,\ \<X,ee^\top\>=0,\ X\in \S^n_+\right\},
%\end{equation}
%
%\begin{equation}\label{BCsmo}
%\min\left\{ f(R):=\frac{1}{2}\< C,RR^\top\>:\ R\in \B_{n,r}\right\}.
%\end{equation}

Consider some singular point $R=ae_1^\top,$ where $a\in E^n.$ For any $[0,H]\in \T_{\B_{n,r}}(R)$ and any $[-a\circ\dd\(HH^\top\),W]\in \T^2_{\B_{n,r}}\(R,[0,H]\).$ Define 
\begin{equation}\label{May_2_15}
\widehat{R}(t,H,W)=\left[ a-\frac{t^2}{2}a\circ\dd\big(HH^\top\big),\ tH+\frac{t^2}{2}W\right].
\end{equation}
From the property of the tangent cone and second order tangent set, we have that 
\begin{equation}\label{May_2_16}
\P_{\B_{n,r}} \big(\widehat{R}(t,H,W)\big)=\widehat{R}(t,H,W)+o(t^2).
\end{equation}

Substituting $\widehat{R}(t,H,W)$ into $f,$ we have that
\begin{equation}\label{May_2_17}
f\big(\widehat{R}(t,H,W)\big)=\frac{a^\top Ca}{2}-\frac{t^2a^\top C\(a\circ\dd\(HH^\top\)\)}{2}+\frac{t^2\<CH,H\>}{2}+O(t^3).
\end{equation}
(\ref{May_2_16}) implies that $W$ has little influence on the function value and we can simply choose $W=0.$ Combining
(\ref{May_2_16}) and (\ref{May_2_17}), we have that
\begin{equation}\label{May_2_18}
f\(\P_{\B_{n,r}} \big(\widehat{R}(t,H,0)\big)\)=f(R)+t^2F(H)+o(t^2),
\end{equation}
where 
\begin{equation}\label{May_11_1}
F(H):=-\frac{a^\top C\(a\circ\dd\(HH^\top\)\)}{2}+\frac{\<CH,H\>}{2}=\frac{1}{2}\< \(C-\DD\big(Caa^\top\big)\)H,H\>.
\end{equation}

From (\ref{May_2_18}), we know that in order to reduce the function value along a certain direction as fast as possible, we should choose $H\in \T_{\B_{n,r}}(R)$ such that the coefficient of $t^2$ in (\ref{May_2_18}) is as small as possible. This introduces the following problem:

\begin{equation}\label{singopt}
\min\Bigg\{ F(H):\ e^\top H=0,\ a^\top\dd\big(HH^\top\big)=0,\ \|H\|^2=n,\ H\in \R^{n\times (r-1)} \Bigg\}.
\end{equation}
The constraint ensures that $[0,H]\in \T_{\B_{n,r}}(R).$ Also, we additionally fix the norm of $H$ to be $n$ because we only care about the direction rather than the length of the vector. The following lemma says that any feasible solution with a negative value provides us with a descent direction that can escape from a singular point.

\begin{lem}\label{escsing}
Suppose $r>1,$ $R=ae_1^\top\in \R^{n\times r}$ is a singular point of $\B_{n,r}$ for some $a\in E^n$ and $f$ is defined as in (\ref{BCsmo}). If problem (\ref{singopt}) has some feasible solution $H\in \R^{n\times (r-1)}$ such that its objective function value is negative. Then there exists some $\alpha<0$ such that 
\begin{equation}\label{decrease}
f\(\P_{\B_{n,r}} \big(\widehat{R}(t,H,0)\big)\)=f(R)+\alpha t^2+o(t^2).
\end{equation}
Moreover, $\P_{\B_{n,r}} \big(\widehat{R}(t,H,0)\big)$ is a smooth point of $\B_{n,r}$ for any nonzero $t$ that is sufficiently small. 
\end{lem}

\begin{proof}
For any feasible solution $H$ of problem (\ref{singopt}), the constraints of (\ref{singopt}) and (\ref{sgtcone}) implies that $[0,H]\in \T_{\B_{n,r}}(R).$ Therefore, from (\ref{May_2_18}) and (\ref{May_11_1}), we have that
\begin{equation}\label{May_2_20}
f\(\P_{\B_{n,r}} \big(\widehat{R}(t,H,0)\big)\)=f(R)+F(H) t^2+o(t^2).
\end{equation}
Thus, if $F(H)<0$, then (\ref{decrease}) holds by defining $\alpha:=F(H).$

Now we move on to prove the smoothness of $\P_{\B_{n,r}} \big(\widehat{R}(t,H,0)\big)$ on the condition that $F(H)<0.$ First, since $F(H)<0,$ from (\ref{decrease}), we have that there exists $\epsilon>0$ such that $f\big(\P_{\B_{n,r}} \big(\widehat{R}(t,H,0)\big)\big)<f(R)$ for any $0<|t|<\epsilon.$ Note that for any singular point $R_1$ of $\B_{n,r},$ if $R_1=ab^\top$ for some $b\in {\rm S}^{r-1},$ then we have that $f(R_1)=f(R).$ This means that 
$\P_{\B_{n,r}} \big(\widehat{R}(t,H,0)\big)$ can not be $R_1$ for any $0<|t|<\epsilon.$ On the other hand, if 
$R_1\notin a({\rm S^{r-1}})^\top$, then we have that $R_1=\hat{a}b^\top$ for some $\hat{a}\in E^n$ and $b\in {\rm S}^{r-1}$ such that $|a^\top \hat{a}|\leq n-1.$ Hence
\begin{align}
&\| R-R_1\|_F^2=\|R\|^2+\|R_1\|^2-2\<R,R_1\>=2n-2\<ae_1^\top,\hat{a}b^\top\>=2n-2a^\top \hat{a}e_1^\top b\notag \\
&\geq 2n-2|a^\top \hat{a}|\geq 2n-2(n-1)=2.\label{May_2_21}
\end{align}
From (\ref{May_2_15}) and (\ref{May_2_16}), we have that $\P_{\B_{n,r}} \big(\widehat{R}(t,H,0)\big)=R+O(t).$ This together with (\ref{May_2_21}) implies that there exists $\epsilon_1>0$ such that $\P_{\B_{n,r}} \big(\widehat{R}(t,H,0)\big)$ can not be $R_1$ for any $|t|<\epsilon_1.$ Thus, $\P_{\B_{n,r}} \big(\widehat{R}(t,H,0)\big)$ is smooth for any $0<|t|<\min\left\{\epsilon_1,\epsilon_2\right\}.$
\end{proof}

Lemma~\ref{escsing} tells us that we can find a descent direction at a singular point by solving (\ref{singopt}). With Lemma~\ref{escsing}, we make the following definition.

\begin{defi}\label{escdefi}
Suppose $r>1,$ $R=ae_1^\top\in \R^{n\times r}$ is a singular point of $\B_{n,r}$ for some $a\in E^n$ and $f$ is defined as in (\ref{BCsmo}). $H\in \R^{n\times (r-1)}$ is called an escaping direction of $(\ref{BCsmo})$ at $R=ae_1^\top$ if it is a feasible solution of (\ref{singopt}) with a negative objective value.
\end{defi}

Note that (\ref{singopt}) is a non-convex problem which may have spurious local minima. Moreover, if there does not exist a feasible solution with a negative objective function value for (\ref{singopt}), then what can we tell about $R$? In order to handle these two problems, we consider the following SDP relaxation of (\ref{singopt}):
\begin{equation}\label{singSDP}
\min\Bigg\{ \<\(C-\DD\big(Caa^\top\big)\),X\>: 
\begin{array}{l}
Xe=0,\ \<X,ee^\top\>=0,\ \< \dd(a),X\>=0,\ 
 \\
\<I,X\>=n, X\in \S^n_+
\end{array}
\Bigg\}. 
\end{equation}
Note that we add the redundant constraint $Xe=0$ similarly as in problem (\ref{SDP1}). Let $J=PP^\top,$ where $P\in \St(n,n-1).$ For any feasible solution $X$ of (\ref{singSDP}), we have that $X=JXJ=PP^\top XPP^\top.$ By introducing $Y:=P^\top XP,$ problem (\ref{singSDP}) is equivalent to the following SDP problem with a smaller size.
\begin{equation}\label{singSDP1}
\min\left\{ \<C_P,Y\>:\ \< P^\top\dd(a)P,Y\>=0,\ \<I,Y\>=n,\ Y\in \S^{n-1}_+\right\},
\end{equation}
where $C_P:=P^\top\(C-\DD\(Caa^\top\)\)P\in \S^{n-1}.$ The Lagrangian dual problem of (\ref{singSDP1}) is as follows:
\begin{equation}\label{singSDPd}
\max\left\{ ny_1:\ C_P-y_1I-y_2 P^\top\dd(a)P\in \S^{n-1}_+,\ y_1,y_2\in \R \right\}.
\end{equation}
Note that we have 
\begin{equation}\label{May_4_1}
\<P^\top \dd\(a\)P,I\>=\<\dd(a),J\>=\<\dd\(a\),\dd\(J\)\>=\frac{(n-1)}{n}\<a,e\>=0.
\end{equation}
(\ref{May_4_1}) implies that $Y=\frac{nI}{n-1}$ is a strict feasible solution of (\ref{singSDP1}). Also, we may choose $y_1$ to be sufficiently negative to make $\(y_1,y_2\)$ strictly feasible for (\ref{singSDPd}). The Slater's conditions for primal and dual problem imply that the duality gap is zero and the optimal solution sets are non-empty for (\ref{singSDP1}) and (\ref{singSDPd}). Thus, problem (\ref{singSDP1}) has a KKT solution. It is easy to check that $P^\top aa^\top P$ is a feasible solution of $(\ref{singSDP1})$ with optimal value $0.$ Thus, the optimal value of (\ref{singSDP1}) is non-positive. The following theorem tells that we can check the global optimality of a singular point and escape from a non-optimal singular point of (\ref{BCsmo}) by solving the SDP problem (\ref{singSDP1}). 

\begin{theo}\label{esctheo}
Suppose $r>1,$ $R=ae_1^\top\in \R^{n\times r}$ is a singular point of $\B_{n,r}$ for some $a\in E^n.$ Let $\(Y,\lambda_1,\lambda_2\)$ be a KKT solution of (\ref{singSDP1}).
\begin{itemize}
\item [(i)] If $\<C_P,Y\>=0,$ then $aa^\top$ is a global optimal solution of (\ref{BCSDP}).
\item [(ii)] If $\<C_P,Y\><0,$ and $Y=HH^\top$ for some $H\in \R^{(n-1)\times (r-1)},$ then $PH$ is an escaping direction of (\ref{BCsmo}) at $R=ae_1^\top.$
\end{itemize}
\end{theo}

\begin{proof}
We first prove (i). $\<C_P,Y\>=0$ implies that the optimal value of (\ref{singSDP1}) is zero. Because $P^\top aa^\top P$ is also a feasible solution of $(\ref{singSDP1})$ with optimal value $0,$ $\(P^\top aa^\top P,\lambda_1,\lambda_2\)$ is a KKT solution of (\ref{singSDP1}). We have the following KKT conditions hold for $\(P^\top aa^\top P,\lambda_1,\lambda_2\):$
\begin{align}
& \< P^\top \dd(a)P,P^\top aa^\top P\>=0,\ \<I,P^\top aa^\top P\>=n,\ P^\top aa^\top P\in \S^{n-1}_+.\label{KKTp}\\
& C_P-y_1 I-y_2 P^\top\dd(a)P\in \S^{n-1}_+.\label{KKTd}\\
& \< C_P-y_1 I-y_2 P^\top\dd(a)P,P^\top aa^\top P\>=0.\label{KKTc}
\end{align}
Note that (\ref{KKTp}) holds for any $a\in E^n.$ Because $C_P=P^\top\(C-\DD\(Caa^\top\)\)P,$ (\ref{KKTd}) implies that:
\begin{equation}\label{May_5_0.5}
P\(P^\top\(C-\DD\big(Caa^\top\big)\)P-y_1I-y_2P^\top \dd(a)P\)P^\top\in \S^n_+.
\end{equation}
After using $PP^\top=J$ and $J^2=J$ in (\ref{May_5_0.5}), we get
\begin{equation}\label{May_5_1}
J\(C-\DD\big(Caa^\top\big)-y_1I-y_2 \dd(a)\)J\in \S^n_+,
\end{equation}
and (\ref{KKTc}) implies that
\begin{equation}\label{May_5_2}
\< J\(C-\DD\big(Caa^\top\big)-y_1I-y_2 \dd(a)\)J, aa^\top \>=0.
\end{equation}
Define $\mu:=\dd\(Caa^\top\)-y_1e-y_2 a.$ Combining (\ref{May_5_1}), (\ref{May_5_2}) with feasibility of $aa^\top$, we have that:
\begin{align}
& \dd\big(aa^\top\big)=e,\ aa^\top e=0,\ \<aa^\top,ee^\top\>=0,\ aa^\top\in \S^n_+.\\
& J\( C-\dd(\mu)\)J\in \S^n_+.\\
& \< J\(C-\dd(\mu)\)J, aa^\top \>=0.
\end{align}
The above equations are exactly the KKT condition of (\ref{BCSDP}). Therefore, $aa^\top$ is an optimal solution of (\ref{BCSDP}).

Now we move on to prove (ii). Because $Y$ is feasible for (\ref{singSDP1}), $PYP^\top$ is feasible for (\ref{singSDP}). Also, because $\<C_P,Y\><0,$  we have that 
$\<C-\DD\big(Caa^\top\big),\,PYP^\top\><0.$ Therefore, $PH$ is a feasible solution of (\ref{singopt}) with a negative function value. From Definition~\ref{escdefi}, $PH$ is an escaping direction of (\ref{BCsmo}) at $R=ae_1^\top.$
\end{proof}

Theorem~\ref{esctheo} says that when we come to a singular point $ae_1^\top$ for some $a\in E^n.$ We can first solve another SDP problem (\ref{singSDP1}), which always has a KKT solution. If the optimal value for (\ref{singSDP1}) is zero, then $aa^\top $ is already an optimal solution of (\ref{BCSDP}). Moreover, from the proof of Theorem~\ref{esctheo}, we can get the optimal dual variable of (\ref{BCSDP}) directly from the optimal dual variable of (\ref{singSDP1}). If the optimal value for (\ref{singSDP1}) is negative and the optimal solution has rank $\leq r-1$, then we can construct an escaping direction from the optimal solution of (\ref{singSDP1}).

One may wonder  whether it is too expensive to escape from a singular point by solving an SDP with the matrix variable of size $(n-1)\times (n-1).$ Actually, problem (\ref{singSDP1}) is simple because there are only 2 constraints and there exists an optimal solution of rank $\sqrt{2\cdot 2}=2.$ This implies that we can use the Burer and Monteiro factorization to solve (\ref{singSDP1}) efficiently. Moreover, we can terminate the algorithm in advance when its primal feasibility is small enough and its function value is negative because an escaping direction does not require that $H$ to be a minimizer of (\ref{singopt}).

%Another point that needs to be addressed is
Next we discuss
 how to compute the retraction 
$\P_{\B_{n,r}} \big(\widehat{R}(t,H,0)\big)$ for a singular point $R=ae_1^\top$ such that $a\in E^n.$ In section 3.3, we have only discussed the retraction around a smooth point. When it comes to a singular point, suppose $H\in \R^{n\times (r-1)}$ is an escaping direction of (\ref{BCsmo}). For the maximum eigenvalue function $\lambda_{\max}(\cdot),$ since $n$ is the only nonzero eigenvalue of $RR^\top=aa^\top,$ $\lambda_{\max}$ is differentiable at $RR^\top$ and the gradient is given by:
\begin{equation}\label{May_5_3}
\nabla \lambda_{\max}\big( RR^\top\big)=\frac{1}{n} aa^\top.
\end{equation}
The following proposition tells that the geometric median approach in 
{subsection \ref{subsec-retraction}} also works for a singular point.
\begin{prop}\label{singgm}
Suppose $r>1,$ $R=ae_1^\top\in \R^{n\times r}$ is a singular point of $\B_{n,r}$ for some $a\in E^n$, and $H\in \R^{n\times (r-1)}$ is an escaping direction of $(\ref{BCsmo})$ at $R=ae_1^\top.$ There exists $\delta>0$ such that for any $0<|t|<\delta,$ 
there exists $0<\beta<1$ such that $\|\widehat{R}(t,H,0)\|<(1-\beta)\sqrt{n}$ and 
$\dist\big(\widehat{R}(t,H,0),\B_{n,r}\big)<\frac{\beta\sqrt{n}}{\sqrt{n}+1}.$
\end{prop}

\begin{proof}
From (\ref{May_2_15}), 
\begin{align}
\widehat{R}(t,H,0)\widehat{R}(t,H,0)^\top=&\;\;
aa^\top-\frac{t^2}{2}\( a\,\dd\big(HH^\top\big)^\top \dd(a)+\dd(a)\dd\big(HH^\top\big)a^\top \)\notag \\
&+t^2HH^\top+O(t^3).\label{May_5_4}
\end{align}
Thus, from (\ref{May_5_3}), we have that
\begin{equation}\label{May_5_5}
\lambda_{\max}\big( \widehat{R}(t,H,0)\widehat{R}(t,H,0)^\top \big)
=n-t^2 \big(n\|H\|^2-\|a^\top H\|^2\big)+O(t^3),
\end{equation}
where we have used the fact that $a\in E^n$ and so $a^\top a=n,$ $\dd(a)a=e.$
Using the Cauchy-Schwarz inequality on every entry of $a^\top H$, we have that 
\begin{equation}\label{May_5_6}
\|a^\top H\|^2\leq \|a\|^2\|H\|^2=n\|H\|^2.
\end{equation}
The equality is attained in (\ref{May_5_6}) if and only if $H=ab^\top$ for some vector $b\in \R^{r-1}.$ However, this implies that the objective function value of problem (\ref{singopt}) is $0,$ which contradicts to that $H$ is an escaping direction. Thus, we have that $\|a^\top H\|^2<n\|H\|^2.$ From the (\ref{May_5_5}) and relation between eigenvalue and singular value, we have that 
\begin{equation}\label{May_5_7}
\| \widehat{R}(t,H,0)\|_2=\sqrt{n}\(1-\frac{t^2}{2n}\(n\|H\|^2-\|a^\top H\|^2\)\)+O(t^3).
\end{equation}
Also, from (\ref{May_2_16}), we have that 
\begin{equation}\label{May_5_8}
\dist\big( \widehat{R}(t,H,0),\B_{n,r}\big)=o(t^2).
\end{equation}
(\ref{May_5_7}) together with (\ref{May_5_8}) implies that if we choose $\beta_t:=\frac{t^2}{4n}\(n\|H\|^2-\|a^\top H\|^2\)>0,$ then there exists $\delta>0$ such that for any $0<|t|<\delta,$ such that $\|\widehat{R}(t,H,0)\|<(1-\beta_t)\sqrt{n}$ and 
$\dist\big(\widehat{R}(t,H,0),\B_{n,r}\big)<\frac{\beta_t\sqrt{n}}{\sqrt{n}+1}.$
\end{proof}

Proposition~\ref{singgm} and Proposition~\ref{analysis} imply that our retraction technique is also useful when escaping from a singular point.

%%%%%%%%%%%%%%%%%%%%%%%%%%%%%%Proof%%%%%%%%%%%%%%%%%%%%%%%%%%%%%
\section{Proof details of some results}

\subsection{Proof of Proposition~\ref{pcone}}
\begin{proof}
We only have to prove that $\T_{\B_{n,r}}(R)\subset \widetilde{\T}\subset \T^i_{\B_{n,r}}(R)$ because we already have $\T^i_{\B_{n,r}}(R)\subset \T_{\B_{n,r}}(R).$ \\
\noindent{\bf Step 1.} $\T_{\B_{n,r}}(R)\subset \widetilde{\T}.$\\
For any $U\in \T_{\B_{n,r}}(R),$ from definition (\ref{ctancone}) there exists $t_k\downarrow 0$ such that $\dd\( R+t_k U,\B_{n,r}\)=o(t_k).$ Let $U=[h,H],$ where $h\in \R^n$ and $H\in \R^{n\times (r-1)}.$ Define 
$$R_k\in \arg\min\left\{ \| S-\(R+t_k U\)\|_F^2:\ S\in \B_{n,r}\right\}$$ 
From the property of the tangent cone, we have that 
$$R_k=R+t_k U+o(t_k)=\( a+t_kh+o(t_k),\ t_kH+o(t_k)\).$$
For $k$ sufficiently large, we also have that 
\begin{equation}\label{May_1_1}
R_k=\(\dd\(a\)\sqrt{e-\dd\(\( t_kH+o(t_k)\)\(t_kH+o(t_k)\)^\top\)},\ t_kH+o(t_k)\),
\end{equation}
where the square root above means componentwise operations. Note that in (\ref{May_1_1}), the first column of $R_k$ comes from the fact that $\dd\(R_kR_k^\top\)=e$ and $R_k$ is close enough to $R.$ The Taylor expansion of (\ref{May_1_1}) gives:
\begin{equation}\label{May_1_1.5}
R_k=\(\dd\(a\)\Big(e-\frac{t_k^2 \dd\(HH^\top\)}{2}\Big)+o(t_k^2),\ t_kH+o(t_k)\).
\end{equation}
Comparing (\ref{May_1_1.5}) with $R_k=\(a+t_k h+o(t_k),\ t_k H+o(t_k)\),$ we have that 
$h=0$ and so $U=[0,H].$ Also, because $e^\top R_k=0,$ we have that $e^\top H=0$ and $a^\top \dd\(HH^\top\)=0.$ Therefore, we have that $U\in \widetilde{\T}$, and hence
$\T_{\B_{n,r}}(R)\subset \widetilde{\T}.$\\

\noindent{\bf Step 2.} $\widetilde{\T}\in \T^i_{\B_{n,r}}\(R\).$\\
For any nonzero $U=[0,H]\in \widetilde{\T},$ we have that $e^\top H=0$ and $a^\top\dd\( HH^\top\)=0.$ Consider the following mapping $\widehat{R}:\R\times \R^{n\times (r-1)}\rightarrow \R^{n\times r}$:
\begin{equation}\label{May_1_4}
\widehat{R}(t,W):=\( \dd\(a\)\sqrt{ e-\dd\( \Big(tH+\frac{t^2 W}{2}\Big) \Big(tH+\frac{t^2 W}{2}\Big)^\top\)},\; tH+\frac{t^2W}{2}\).
\end{equation}
It is easy to see that $\widehat{R}(t,W)\in \OB_{n,r}$ for $t$ small enough. Our next step is to find $W$ that satisfies $e^\top \widehat{R}(t,W)=0.$ In order to do this, consider the Taylor expansion of $e^\top \widehat{R}(t,W)$ with respect to $t$:
\begin{align}
&e^\top \widehat{R}(t,W)=\( a^\top\sqrt{ e-\dd\( \Big(tH+\frac{t^2 W}{2}\Big) \Big(tH+\frac{t^2 W}{2}\Big)^\top\)},\; te^\top H+\frac{t^2e^\top W}{2}\)\label{May_1_4.5}\\
&=\(a^\top\( e-\frac{1}{2}\dd\(\Big(tH+\frac{t^2 W}{2}\Big)\Big(tH+\frac{t^2W}{2}\Big)^\top\) 
-\frac{1}{8}t^4\dd^2\big(HH^\top\big)
+t^5g(t,W)\),\frac{t^2 e^\top W}{2}\)\notag \\
&=\( -\frac{t^3a^\top \dd\(HW^\top\)}{2}-\frac{t^4 a^\top \dd\(WW^\top\)}{8}-\frac{t^4 a^\top \dd^2\(HH^\top\)}{8}+t^5a^\top g(t,W),\frac{t^2e^\top W}{2}\)\notag,
\end{align}
where $g(t,W):\R\times \R^{n\times (r-1)}\rightarrow \R^{n\times r}$ is a smooth mapping when $W$ is bounded and $t$ is small enough. Note that we have used $e^\top H=0$ in the second equality and $a^\top\dd\( HH^\top\)=0,$ $e^\top a=0$ in the third equality. Consider the following two cases:\\

\noindent{\bf Case 1.} $H\notin a\R^{1\times (r-1)}$\\
In this case, consider the following mapping:
\begin{equation}\label{May_1_5}
F(t,W):=\Big( -\frac{a^\top \dd\(HW^\top\)}{2}-\frac{t a^\top \dd\(WW^\top\)}{8}-\frac{t a^\top \dd^2\(HH^\top\)}{8}+t^2a^\top g(t,W),\frac{e^\top W}{2}\Big).
\end{equation}
We have that $F(0,0)=0$ and $\D_W F(0,0)[\triangle W]=\(-\frac{1}{2}{a^\top \dd\(H\triangle W^\top\)},\frac{1}{2}{e^\top \triangle W}\).$ For any $\(\alpha,\lambda\)\in \R\times \R^{r-1},$ $\D_W F(0,0)^*\(\alpha,\lambda\)=-\frac{\alpha}{2}\dd\(a\)H+\frac{e\lambda^\top}{2}=\dd\(a\)\big( -\frac{\alpha}{2}H+\frac{a\lambda^\top}{2}\big).$ 
Because $H\notin a\R^{1\times (r-1)},$ we have that $\D_W F(0,0)^*\(\alpha,\lambda\)=0$ if and only if $\(\alpha,\lambda\)=0.$ This implies that $\D_W F(0,0)$ is surjective (or full rank). From Lemma~\ref{implicit}, we have that there exists $\delta>0$ and a continuously differentiable mapping $W(t):(-\delta,\delta)\rightarrow \R^{n\times (r-1)}$ such that $W(0)=0$ and $F(t,W(t))=0$ for $t$ sufficiently small. From (\ref{May_1_4.5}), $e^\top \widehat{R}\(t,W(t)\)=0$ for $t$ sufficiently small. This implies that when $t$ is small enough, $\widehat{R}(t,W(t))\in \B_{n,r}.$ Because $\dist\big( \widehat{R}(t,W(t)),R+tU\big)=O(t^2),$ we have that 
$\dist\big(R+tU,\B_{n,r}\big)=O(t^2).$ Thus, $U\in \T^i_{\B_{n,r}}\(R\).$\\

\noindent{\bf Case 2.} $H=a\lambda^\top$ for some $\lambda\in \R^{r-1}.$\\
In this case, we choose $W\equiv 0.$ Because $\dd\big(HH^\top\big)=\|\lambda\|^2e,$ we have that $e^\top \widehat{R}(t,W)=0,$ for $t$ small enough, which in addition implies that 
$\hat{R}\( t,W(t)\)\in \B_{n,r}$ for $t$ sufficiently small. After a similar analysis as in Case 1, we get $\dist\(R+tU,\B_{n,r}\)=O(t^2).$ Thus, $U\in \T^i_{\B_{n,r}}\(R\).$
\end{proof}

\subsection{Proof of Proposition~\ref{sectan}}
\begin{proof}
We only have to prove $\T^{2}_{\B_{n,r}}(R,H)\subset \widetilde{\T^2}(R,H)\subset \T^{i,2}_{\B_{n,r}}(R,H),$ since we already have $\T^{i,2}_{\B_{n,r}}(R,H)\subset \T^{2}_{\B_{n,r}}(R,H).$\\
\noindent{\bf Step 1} $\T^{2}_{\B_{n,r}}(R,H)\subset \widetilde{\T^2}(R,H).$ \\
For any $W=[w,W_1]\in \T^{2}_{\B_{n,r}}(R,H),$ where $w\in \R^n$ and $W_1\in \R^{n\times (r-1)},$ there exists $t_k\downarrow 0$ such that 
$\dist\big( R+t_kH+\frac{t_k^2W}{2},\B_{n,r}\big)=o(t_k^2).$ 
Let $R_k\in \arg\min\left\{ \|S-\big(R+t_k H+\frac{t_k^2W}{2}\big)\|^2:\ S\in \B_{n,r}\right\},$ we have that $R_k$ has the following form:
\begin{equation}\label{May_2_4}
R_k=\begin{pmatrix}\dd\(a\)\sqrt{e-\dd\Big(\( t_kH_1+\frac{t_k^2W_1}{2}+t_k^2g(t_k)\)\(t_kH_1+\frac{t_k^2W_1}{2}+t_k^2g(t_k)\)^\top\Big)},&t_kH_1+\frac{t_k^2W_1}{2}+t_k^2g(t_k)\end{pmatrix},\notag
\end{equation}
where $g(t_k)\in \R^{n\times (r-1)}$ is some matrix sequence such that $g(t_k)=o(1).$
The Taylor expansion of the above formula gives:
\begin{align}\label{May_2_5}
&R_k=\Bigg(\dd\(a\)\Big( e-\frac{t_k^2\dd\(H_1H_1^\top\)}{2}-\frac{t_k^3\dd\(H_1W_1^\top\)}{2}-\frac{t_k^4\dd\(W_1W_1^\top\)}{8}-\frac{t_k^4\dd\(H_1H_1^\top\)}{8}\notag \\
&\qquad\qquad -t_k^3\dd\big(H_1g(t_k)^\top\big)\Big)+o(t_k^4),\ t_kH_1+\frac{t_k^2W_1}{2}+t_k^2g(t_k)\Bigg).
\end{align}
Comparing (\ref{May_2_5}) with $R_k=\begin{pmatrix} a+\frac{t_k^2 w}{2}+o(t_k^2),& t_k H_1+\frac{t_k^2 W_1}{2}+o(t_k^2)\end{pmatrix},$ we have that 
\begin{equation}\label{May_2_6}
w=-\dd\(a\)\dd\big(H_1H_1^\top\big).
\end{equation}
 Also, because $e^\top R_k=0,$ we have that 
 \begin{equation}\label{May_2_7}
 a^\top \dd\big(H_1 W_1^\top\big)=0,\ e^\top W_1=0,\ e^\top g(t_k)=0.
\end{equation}
\noindent{\bf Case 1.1} $H_1\notin a\R^{1\times (r-1)}.$\\
From (\ref{May_2_6}) and (\ref{May_2_7}), we have that $W:=[w,W_1]\in \widetilde{\T^2}(R,H)$ and so $\T^2_{\B_{n,r}}(R,H)\subset \widetilde{\T^2}(R,H).$\\
\noindent{\bf Case 1.2} $H_1=a\lambda^\top$ for some $\lambda\in \R^{r-1}.$\\
From (\ref{May_2_7}), we have that 
\begin{align}\label{May_2_8}
&e^\top \dd\(a\)\dd\big(H_1 g(t_k)^\top\big)=a^\top \dd\big(H_1 g(t_k)^\top\big)=\<\dd\(a\)H_1,g(t_k)\>\notag \\
&=\<\dd(a)a\lambda^\top,g(t_k)\>=\< e\lambda^\top,g(t_k)\>=\< \lambda^\top,e^\top g(t_k)\>=0.
\end{align}
Combining (\ref{May_2_8}) and $e^\top R_k=0,$ we have that $a^\top \dd\(W_1W_1^\top\)+a^\top\dd\(H_1H_1^\top\)=0.$ Note that $a^\top\dd\(H_1H_1^\top\)=a^\top e \|\lambda\|^2=0.$ This implies that $a^\top \dd\(W_1W_1^\top\)=0.$ This together with (\ref{May_2_6}) and (\ref{May_2_7}) implies that $W=[w,W_1]\in \widetilde{\T^2}(R,H)$.\\

\noindent{\bf Step 2.} $\widetilde{\T^2}(R,H)\subset \T^{i,2}_{\B_{n,r}}(R,H).$\\
For any $W=[-a\circ\dd\(H_1H_1^\top\),W_1]\in \widetilde{\T^2}(R,H),$ where $W_1\in \R^{n\times (r-1)}$, consider the following mapping:
$\widehat{R}(t,X):\R\times \R^{n\times (r-1)}\rightarrow \R^{n\times r}.$
\begin{equation}\label{May_2_9}
\widehat{R}(t,X)=\( \dd\(a\)\sqrt{ e-\dd\Big( \Big(tH_1+\frac{t^2 W_1}{2}+t^3X\Big) 
\Big(tH_1+\frac{t^2 W_1}{2}+t^3X\Big)^\top\Big)},tH_1+\frac{t^2W_1}{2}+t^3X\).\notag
\end{equation}
It is easy to see that $\hat{R}(t,X)\in \OB_{n,r}$ when $X$ is bounded and $t$ is small enough. Its Taylor expansion gives:
\begin{align}\label{May_2_10}
&\widehat{R}(t,X) =\Bigg( \dd(a)\Bigg[ e-\frac{t^2\dd\(H_1H_1^\top\)}{2}-\frac{t^3\dd\(H_1W_1^\top\)}{2}\notag \\
&\qquad\qquad\qquad-\frac{t^4\dd\(W_1W_1^\top\)}{8}-t^4\dd\big( H_1X^\top\big)
-\frac{t^4}{8}\dd^2\big(H_1H_1^\top\big)\notag \\
&\qquad\qquad\qquad-\frac{t^5\dd\(W_1X^\top\)}{2}-t^5K+t^6g(t,X)\Bigg],\ tH_1+\frac{t^2W_1}{2}+t^3X\Bigg),
\end{align}
where $K\in \R^{n}$ is a constant matrix and $g(t,X):\R\times \R^{n\times (r-1)}\rightarrow \R^{n}$ is a smooth mapping when $W_1$ is bounded and $t$ is small enough.\\
\noindent{\bf Case 2.1} $H_1\notin a\R^{1\times (r-1)}$\\
In this case, we have that 
\begin{equation}\label{May_2_11}
a^\top e=1,\ a^\top \dd\big(H_1H_1^\top\big)=0,\ a^\top\dd\big(H_1W_1^\top\big)=0,\ e^\top H_1=0,\ e^\top W_1=0,
\end{equation}
where the above equalities come from the fact that $H\in \T_{\B_{n,r}}(R)$ and $W\in \widetilde{\T^2}(R,H).$
In order to ensure that $e^\top \widehat{R}(t,X)=0,$ we consider the following mapping:
\begin{align}\label{May_2_12}
&F(t,X):= \Bigg(\frac{-a^\top \dd\(W_1W_1^\top\)}{8}-a^\top \dd\big(H_1X^\top\big)-\frac{a^\top \dd^2\(H_1H_1^\top\)}{8}\notag \\
&\qquad\qquad\qquad -t\frac{a^\top \dd\(W_1X^\top\)}{2}-ta^\top K+t^2 a^\top g(t,X),\  e^\top X\Bigg).
\end{align}
Because $H_1\notin a\R^{1\times (r-1)}$ we have that linear operator 
$$\D F(0,X)[\triangle X]=\( -a^\top \dd\(H_1\triangle X\),\ e^\top \triangle X\)$$
 is surjective. This also implies that there exists $X_0\in \R^{n\times (r-1)}$ such that $F(0,X_0)=0.$ Using Lemma~\ref{implicit}, we have that there exists $\delta>0$ and continuously differentiable mapping $X(t):(-\delta,\delta)\rightarrow \R^{n\times (r-1)}$ such that $X(0)=X_0$ and $F(t,X(t))=0$ for any $t\in (-\delta,\delta)$. This together with (\ref{May_2_11}) implies that $e^\top \widehat{R}(t,X(t))=0$ for $t$ sufficiently small. Since $\dist\(\widehat{R}(t,X),R+tH+\frac{t^2W}{2}\)=O(t^3),$ we have that $W\in \T^{i,2}_{\B_{n,r}}(R,H).$ \\
 
 \noindent{\bf Case 2.2} $H_1=a\lambda^\top,$ for some $\lambda\in \R^{r-1}$ 
 and $W_1\notin a\R^{1\times(r-1)}.$\\
In this case, we have that the following equalities hold:
\begin{align}\label{May_2_13}
&a^\top e=1,\ a^\top \dd\big(H_1H_1^\top\big)=0,\ a^\top \dd\big(W_1W_1^\top\big)=0,\ e^\top H_1=0,\notag \\
&e^\top W_1=0,\ a^\top \dd^2\big(H_1H_1^\top\big)=0,\ a^\top \dd\big(H_1W_1^\top\big)=0.
\end{align}
In addition, we also have that for any $X\in \R^{n\times (r-1)}$ such that $e^\top X=0,$ the following equalities hold
\begin{equation}\label{May_2_13.5}
a^\top \dd\big(H_1X^\top\big)=0.
\end{equation}
The proof of (\ref{May_2_13.5}) is the same as (\ref{May_2_8}).
Consider the following mapping:
\begin{equation}\label{May_2_14}
F(t,X):=\( \frac{-a^\top \dd\(W_1X^\top\)}{2}-a^\top K+ta^\top g(t,X),e^\top X\).
\end{equation}
Similar to case 2.1, by using $W_1\notin a\R^{1\times (r-1)},$ we can prove that there exists $\delta>0$ and a continuously differentiable mapping $X(t)$ in $(-\delta,\delta)$ such that $F(t,X(t))=0$ for any $t\in (-\delta,\delta)$. Then, a similar argument as in case 2.1 shows that $\dist\big(\widehat{R}(t,X),R+tH+\frac{t^2W}{2}\big)=O(t^3)$ and hence $W\in \T^{i,2}_{\B_{n,r}}(R,H).$\\
\noindent{\bf Case 2.3} $H_1=a\lambda_1^\top$ and $W_1=a\lambda_2^\top$ for some $\lambda_1,\lambda_2\in \R^{r-1}.$
In this case, we choose $X\equiv 0.$ It is easy to verify that $\widehat{R}(t,X)\in \B_{n,r}$ for $t$ small enough. Thus, $\dist\big(\widehat{R}(t,X),R+tH+\frac{t^2W}{2}\big)=O(t^3)$ and so $W\in \T^{i,2}_{\B_{n,r}}(R,H).$
\end{proof}
}

\end{document}